\documentclass{amsart}

\usepackage{amsmath,amstext,amssymb,mathrsfs,amscd,amsthm,indentfirst}
\usepackage{amsfonts}
\usepackage{xypic}
\usepackage{enumerate}
\usepackage{verbatim}
\usepackage{pstricks}
\usepackage{graphicx}

\usepackage{hyperref}

\theoremstyle{plain}
\newtheorem{theorem}{Theorem}[section]
\newtheorem{proposition}[theorem]{Proposition}
\newtheorem{lemma}[theorem]{Lemma}
\newtheorem{corollary}[theorem]{Corollary}
\newtheorem{MainThm}{Theorem}

\theoremstyle{definition}
\newtheorem{definition}[theorem]{Definition}
\newtheorem{claim}[theorem]{Claim}
\newtheorem{remark}[theorem]{Remark}

\numberwithin{equation}{section}

\newcommand{\im}{\operatorname{Im}}
\newcommand{\bS}{\mathbb{S}}
\newcommand{\bQ}{\mathbb{Q}}
\newcommand{\bR}{\mathbb{R}}
\newcommand{\bZ}{\mathbb{Z}}
\newcommand{\bC}{\mathbb{C}}

\newcommand{\bK}{\mathbb{K}}
\newcommand{\bG}{\mathbb{G}}
\newcommand{\bH}{\mathbb{H}}
\newcommand{\bN}{\mathbb{N}}

\newcommand{\bL}{\mathbb{L}}
\newcommand{\Harm}{\mathcal{H}}
\newcommand{\cA}{\mathcal{A}}
\newcommand{\cl}{\mathcal{L}}
\newcommand{\cS}{\mathcal{S}}
\newcommand{\cT}{\mathcal{T}}

\newcommand{\fn}{\mathfrak{n}}
\newcommand{\fg}{\mathfrak{g}}
\newcommand{\fk}{\mathfrak{k}}
\newcommand{\fl}{\mathfrak{l}}

\newcommand{\rank}{\operatorname{rank}}
\newcommand{\sign}{\operatorname{sign}}
\newcommand{\Sp}{\mathrm{Sp}}
\newcommand{\USp}{\mathrm{USp}}
\newcommand{\Ort}{\mathrm{O}}
\newcommand{\SO}{\mathrm{SO}}
\newcommand{\GL}{\mathrm{GL}}
\newcommand{\SL}{\mathrm{SL}}
\newcommand{\U}{\mathrm{U}}

\newcommand{\sm}{sm}
\newcommand{\Diff}{\mathrm{Diff}}
\newcommand{\hAut}{\mathrm{Aut}}
\newcommand{\symb}{\mathrm{symb}}

\newcommand{\ch}{\mathrm{ch}}
\newcommand{\ph}{\mathrm{ph}}

\newcommand{\ind}{\operatorname{ind}}

\newcommand{\map}{\operatorname{map}}
\newcommand{\blockdiff}{\widetilde{\Diff}}
\newcommand{\blockaut}{\widetilde{\hAut}}

\newcommand{\cN}{\mathcal{N}}

\newcommand{\Tor}{\mathfrak{Tor}}

\newcommand{\qh}{\mathrm{qh}}
\newcommand{\ceil}[1]{\lceil #1 \rceil}

\newcommand{\loopinf}[1]{\Omega^{\infty #1}}
\newcommand{\MTSO}{\mathrm{MTSO}}
\newcommand{\KO}{\mathrm{KO}}
\newcommand{\KU}{\mathrm{KU}}
\newcommand{\MTthetan}{\mathrm{MT}\theta^n}
\newcommand{\Gr}{\mathrm{Gr}}

\newcommand{\Emb}{\mathrm{Emb}}

\newcommand{\B}[1]{{\bf #1}}

\newcommand{\Eig}{\mathrm{Eig}}

\newcommand\lra{\longrightarrow}

\newcommand\colim{\operatorname*{colim}}

\newcommand{\hcoker}{/\!\!/}

\title{Torelli spaces of high-dimensional manifolds}
\author{Johannes Ebert}
\thanks{}
\email{johannes.ebert@uni-muenster.de}
\address{Mathematisches Institut der
Westf{\"a}lische Wilhelms-Universit{\"a}t M{\"u}nster\\
Einsteinstr. 62\\
DE-48149 M{\"u}nster\\
Germany }

\author{Oscar Randal-Williams}
\thanks{}
\email{o.randal-williams@dpmms.cam.ac.uk}
\address{Centre for Mathematical Sciences\\
Wilberforce Road\\
Cambridge CB3 0WB\\
UK}



\begin{document}

\begin{abstract}
The Torelli group of a manifold is the group of all diffeomorphisms which act as the identity on the homology of the manifold. In this paper, we calculate the invariant part (invariant under the action of the automorphisms of the homology) of the cohomology of the classifying space of the Torelli group of certain high-dimensional, highly connected manifolds, with rational coefficients and in a certain range of degrees. This is based on Galatius--Randal-Williams' work on the diffeomorphism groups of these manifolds, Borel's classical results on arithmetic groups, and methods from surgery theory and pseudoisotopy theory. As a corollary, we find that all Miller--Morita--Mumford characteristic classes are nontrivial in the cohomology of the classifying space of the Torelli group, except for those associated with the Hirzebruch class, whose vanishing is forced by the family index theorem.
\end{abstract}

\maketitle

\setcounter{tocdepth}{1}
\tableofcontents

\section{Introduction}

The high-dimensional manifolds of the title of this paper are the manifolds $W_{g}^{2n}:= \sharp^g (S^n \times S^n)$ with $n \geq 3$ (though our main result only has content for much larger values of $n$). Let $D^{2n} \subset W_{g}^{2n}$ be a fixed embedded closed disc and let $\Diff(W_{g}^{2n},D^{2n})$ denote the topological group of diffeomorphisms of $W_{g}^{2n}$ which restrict to the identity on a neighbourhood of $D^{2n}$. Alternatively, let $W_{g,1}^{2n} := W_{g}^{2n} \setminus \mathrm{int}(D^{2n})$ be the manifold with boundary obtained by removing the interior of $D^{2n}$, and $\Diff_\partial(W_{g,1}^{2n})$ denote the group of diffeomorphisms of $W_{g,1}^{2n}$ which restrict to the identity near the boundary. Extending diffeomorphisms over $D^{2n}$ by the identity map gives an isomorphism $\Diff_\partial(W_{g,1}^{2n}) \cong \Diff(W_{g}^{2n},D^{2n})$ of topological groups, and we shall use them interchangeably.

The \emph{Torelli group} $\Tor_{g, 1}^{2n} \subset \Diff_\partial (W^{2n}_{g,1})$ is defined to be the subgroup of those diffeomorphisms which induce the identity automorphism of $H_n (W_{g,1}^{2n};\bZ) \cong \bZ^{2g}$. There is a fibre sequence
\begin{equation*}
B \Tor_{g, 1}^{2n} \lra B \Diff_\partial (W^{2n}_{g,1}) \lra B \Gamma(W^{2n}_{g, 1}),
\end{equation*}
where $\Gamma(W^{2n}_{g, 1}) \subset \GL (H_n (W_{g,1}^{2n}; \bZ))$ is a certain arithmetic group, which we shall describe in detail in Section \ref{sec:DefiningGamma}. The rational cohomology of the base and the total space of this fibre sequence are completely known in a stable range, due to the work of Borel (on the cohomology arithmetic groups) and the work of Galatius--Randal-Williams (on the cohomology of diffeomorphism groups).

The aim of this paper is to compute the $\Gamma(W^{2n}_{g, 1})$-invariant part of the rational cohomology of $B \Tor_{g,1}^{2n}$, for large $g$ and $n$. The result is that the behaviour is in a sense as simple as possible.

\begin{MainThm}\label{main-result}
The natural map 
\begin{equation*}
H^* (B \Diff_\partial (W_{g, 1}^{2n}); \bQ)/ \big(\im H^{*>0} (B \Gamma (W_{g, 1}^{2n});\bQ) \big) \lra H^* (B \Tor_{g, 1}^{2n};\bQ)^{\Gamma (W_{g, 1}^{2n})}
\end{equation*}
(the brackets mean ``ideal generated by'') is an isomorphism in degrees $* \leq C_{g}^{2n}$, where $C_{g}^{2n}$ is the largest integer with 
\begin{enumerate}[(i)]
\item $C_{g}^{2n} \leq (g-3)/2$;
\item $2n \geq \max\{2C_{g}^{2n}+7,3C_{g}^{2n}+4\}$.
\end{enumerate}
\end{MainThm}

Throughout the paper we will hold $g$ and $2n$ fixed, and a number $q$ will be called ``in the stable range'' if $q \leq C_{g}^{2n}$. The left hand side of the expression in Theorem \ref{main-result} can be explicitly computed, using the results of \cite{Borel1}, \cite{GRW1}, and \cite{GRW2}, and the result of this calculation is most easily expressed in terms of Miller--Morita--Mumford classes, which we now define. There is a bundle of closed manifolds
$$\pi: (W_{g} \times E\Diff_\partial (W_{g, 1}^{2n}))/\Diff_\partial (W_{g, 1}^{2n}) =: E \lra B\Diff_\partial (W_{g, 1}^{2n}),$$
and we write $T_v E$ for its vertical tangent bundle, an oriented real vector bundle of rank $2n$ over $E$. For each rational characteristic class $c \in H^{k+2n}(B\SO(2n);\bQ)$ of such vector bundles, we define
$$\kappa_c := \pi_!(c(T_v)) \in H^k(B\Diff_\partial (W_{g, 1}^{2n});\bQ),$$
the \emph{generalised Miller--Morita--Mumford class} associated to $c$. We have in particular the Hirzebruch $\cl$-classes $\cl_i \in H^{4i}(B\SO(2n);\bQ)$, and so classes
$$\kappa_{\cl_{a}\cl_b} \in H^{4a+4b-2n}(B\Diff_\partial (W_{g, 1}^{2n});\bQ).$$
These classes may be pulled back to $B \Tor_{g, 1}^{2n}$, where they give $\Gamma(W_{g,1}^{2n})$-invariant elements of cohomology (precisely because they are pulled back from $B\Diff_\partial (W_{g, 1}^{2n})$).

\begin{MainThm}\label{main-result:calc}
The map
$$\bQ\big[\kappa_{\cl_{a}\cl_b} \,\vert \, \tfrac{n+1}{4} \leq a \leq  b \big] \lra H^* (B \Tor_{g, 1}^{2n};\bQ)^{\Gamma (W_{g, 1}^{2n})}$$
is an isomorphism in degrees $* \leq C_g^{2n}$.
\end{MainThm}

\begin{remark}
The reader familiar with mapping class groups of surfaces will have noted that the name ``Torelli group'' is borrowed from that subject. If $\Gamma_g$ is the mapping class group of a closed oriented genus $g$ surface, the Torelli group $\cT_g$ is the kernel of the surjective map $\Gamma_g \to \Sp_{2g}(\bZ)$. The cohomology of the mapping class group is---in the stable range---known, by the Madsen--Weiss theorem \cite{MW}. It is a long-standing question whether the even Mumford classes $\kappa_{2i} \in H^{4i}(\Gamma_g;\bQ)$ remain non-zero when restricted to $\cT_g$ (the odd ones are easily seen to vanish). Theorem \ref{main-result} answers this question in (much) higher dimensions, and can perhaps be considered as suggesting that it is likely that the $\kappa_{2i} \in H^{4i}(\cT_g;\bQ)$ are also non-zero.
\end{remark}

\subsection{Overview of the proof}

The proof of Theorem \ref{main-result} depends on a number of deep results:
\begin{enumerate}[(i)]
\item The computation of the cohomology of $B\Diff_\partial(W_{g,1}^{2n})$ in the stable range, due to Galatius and Randal-Williams.
\item The computation of the cohomology of arithmetic groups, due to Borel.
\item The Atiyah--Singer family index theorem.
\item The space version of the simply connected surgery exact sequence, due to Quinn, and related computations by Berglund and Madsen.
\item Morlet's lemma of disjunction.
\end{enumerate}

In order to make the overall structure of the argument clear, we present an overview. Our argument is a calculation of the Leray--Serre spectral sequence of the fibre sequence 
\begin{equation*}
B \Tor_{g, 1}^{2n} \lra B \Diff_\partial (W^{2n}_{g,1}) \lra B \Gamma(W^{2n}_{g, 1}),
\end{equation*}
with rational coefficients, in the stable range. The group $\Gamma(W^{2n}_{g, 1})$ is an arithmetic group, sitting inside $\Sp_{2g}(\bZ)$ (if $n $ odd) or $\Ort_{g,g}(\bZ)$ (if $n$ even). Borel has calculated the cohomology of such groups (originally with real coefficients, because he uses differential forms, but this implies the rational statement), and his result may be stated by saying that a certain map 
\begin{equation*}
\beta_{g}^{2n} : B \Gamma(W^{2n}_{g, 1}) \lra \Omega^{k_n}_{0} B\Ort := \begin{cases}
B\Ort & \text{if $n$ is even}\\
\Omega^{2}_{0} B\Ort & \text{if $n$ is odd}
\end{cases}
\end{equation*}
induces an isomorphism on rational cohomology in a range of degrees that depends linearly on $g$.

The stable homology of $B \Diff_\partial (W_{g,1}^{2n})$ has been computed by Galatius and Randal-Williams, as we have already mentioned. Their result is that a certain map 
\begin{equation*}
\alpha_{g}^{2n} : B \Diff_\partial(W_{g,1}^{2n}) \lra \Omega^{\infty}_{0} \MTthetan
\end{equation*}
to the infinite loop space of a certain Thom spectrum is a homology equivalence (again, in a range depending linearly on $g$). We will define an infinite loop map
$$\mathrm{symb} : \Omega^{\infty}_{0}\MTthetan \lra \Omega^{k_n}_{0} B\Ort$$
and we will construct, using the Atiyah--Singer family index theorem, a map of fibration sequences (see Theorem \ref{mt-spectra-index} and the discussion in \S \ref{section:final-argument})
\begin{equation}\label{main-diagram}
\begin{gathered}
\xymatrix{
B \Tor_{g,1}^{2n} \ar[r] \ar[d] & B \Diff_\partial(W_{g,1}^{2n}) \ar[r]^{\zeta} \ar[d]^{\alpha_{g}^{2n}} & B \Gamma (W_{g,1}^{2n})\ar[d]^{\beta_{g}^{2n}}\\
\loopinf{} F \ar[r] & \Omega^{\infty}_{0}\MTthetan_{\bQ} \ar[r]^{\mathrm{symb}} & (\Omega^{k_n}_{0} B\Ort)_{\bQ}.
}
\end{gathered}
\end{equation}
The $\bQ$-subscript denotes rationalization in the sense of homotopy theory.
Considering the rational cohomological Leray--Serre spectral sequences, we get a morphism of spectral sequences, which converges to an isomorphism in the stable range (by Galatius and Randal-Williams). The base term in an isomorphism in the stable range (by Borel). We can conclude, by Zeeman's comparison theorem, that the map $H^* ( \loopinf{} F;\bQ) \to H^{*}(B \Tor_{g, 1}^{2n};\bQ)^{\Gamma (W_{g,1}^{2n})}$ is an isomorphism in the stable range, \emph{provided that the $E^2$-pages of the spectral sequences both have a product structure}. The lower fibration sequence is a sequence of infinite loop spaces and infinite loop maps and thus its rational Leray--Serre spectral sequence has a product structure. Its rational cohomology is easy to compute and $H^* (\loopinf{} F; \bQ)$ is isomorphic to the algebra showing up in the left hand side of Theorem \ref{main-result} (in a stable range). In addition, from this argument we may conclude that the rational Leray--Serre spectral sequence of $B \Diff_\partial(W_{g,1}
^{2n}) \to B \Gamma (W_{g,1}^{2n})$ collapses in a stable range of degrees.

What is missing is the product structure on the Leray--Serre spectral sequence for the top fibration sequence, and this is where the other ingredients enter. We call a representation of an arithmetic group $\Gamma$ on a rational vector space $U$ \emph{arithmetic} if $U$ is finite-dimensional and if the representation extends to a holomorphic representation of $G$ on $U \otimes_\bQ \bC$, where $G$ is the ambient complex Lie group of the arithmetic group $\Gamma$. It is a theorem of Borel that for an arithmetic representation $U$, there is an isomorphism
$$H^* (\Gamma;U) \cong H^* (\Gamma;\bQ) \otimes U^{\Gamma}$$
in a range of degrees. If we know that $H^q (B \Tor_{g, 1}^{2n}; \bQ)$ is an arithmetic $\Gamma (W_{g,1}^{2n})$-representation, it follows that the Leray--Serre spectral sequence of the top fibration sequence has a product structure.

How do we prove arithmeticity of this representation? This is done in two steps, the first of which is essentially due to Berglund--Madsen \cite{BM}. In this step, the {block diffeomorphism group} $\blockdiff_\partial(W_{g,1}^{2n})$ is related to surgery theory. More specifically, there is a map
$$\Gamma(W_{g,1}^{2n})\hcoker\blockdiff_\partial(W_{g,1}^{2n}) \lra \map (W_{g,1}^{2n}/\partial W_{g,1}^{2n}, \mathrm{G}/\Ort),$$
which is injective on rational homotopy groups. We show that the induced map on cohomology is one of $\Gamma(W_{g,1}^{2n})$-modules, and as $\mathrm{G}/\Ort$ is an infinite loop space with well-known rational homotopy groups, we find that the cohomology groups $H^q (\Gamma(W_{g,1}^{2n})\hcoker\blockdiff_\partial(W_{g,1}^{2n});\bQ)$ are subquotients of an arithmetic representation, hence arithmetic. (A closer analysis, as carried out in \cite{BM}, gives a lot more information, but for our purpose this is not necessary.) The main tool for this step is the space version of the surgery exact sequence due to Quinn.

The relation between block diffeomorphisms and actual diffeomorphisms is via Morlet's lemma of disjunction. The result is that in the concordance stable range, the map
$$B \Tor_{g, 1}^{2n} \simeq \Gamma(W_{g,1}^{2n})\hcoker\Diff_\partial(W_{g,1}^{2n}) \lra \Gamma(W_{g,1}^{2n})\hcoker\blockdiff_\partial(W_{g,1}^{2n})$$
is a rational homology isomorphism. 
 
\begin{remark}
Let us return briefly to mapping class groups of surfaces and the ordinary Torelli group. One might ask whether a similar argument to that of this paper could potentially work in that case as well.
The argument outlined above shows that \emph{if} one knows that $H^* (\cT_g;\bQ)$ is finite-dimensional and the action of the symplectic group on it is arithmetic, then the answer to the question is positive. However, $H^* (\cT_g;\bQ)$ is in general not finitely generated, as $\cT_2$ is a free group on infinitely many generators (Mess \cite{Mess}). Moreover, for each sufficiently large $g$, the dimension of $H^* (\cT_g; \bQ)$ is known to be infinite (Akita \cite{Aki}); though its homological dimension is finite by Teichm{\"u}ller theory. On the other hand, $H^1 (\cT_g;\bQ)$ is known for $g \geq 3$ by work of Johnson \cite{John}, and it is an arithmetic $\Sp_{2g}(\bZ)$-representation (of degree $3$). It does not seem to be unreasonable to conjecture that $H^q (\cT_g; \bQ)$ is an arithmetic $\Sp_{2g}(\bZ)$-representation of degree $q$ in some range of degrees that grows with $g$.
\end{remark}

\subsection{Structure of the paper}

In Section \ref{section:gmtw-theory-index-theory}, the diagram \eqref{main-diagram} is developed, where the main role is played by the relation between Madsen--Tillmann spectra and the Atiyah--Singer index theorem that has been described in \cite{Eb}. In that section, we also compute the cohomology of the bottom sequence of diagram \eqref{main-diagram}. Section \ref{section:arithmetic-groups} is a survey on several aspects of Borel's work on cohomology of arithmetic groups. The purpose is twofold; Borel's result shows that the right vertical map in the diagram \eqref{main-diagram} is a (rational) homology isomorphism. The second purpose is to discuss the statement that the cohomology with coefficients in an arithmetic representation cuts down to the invariant part. This involves the (purely Lie-algebraic) computation of a constant that has not been made explicit by Borel. In Section \ref{section:rational-homotopy}, we extract from \cite{BM} those results from surgery theory that are necessary to deduce the 
arithmeticity of $H^* (B \Tor_{g, 1}^{2n};\bQ)$ as a $\Gamma(W_{g,1}^{2n})$-representation. We also describe the relation between diffeomorphisms and block diffeomorphisms. In Section \ref{section:final-argument} we put everything together and give the proof of Theorem \ref{main-result}. Finally, in Section \ref{sec:InvThy} we include a brief discussion comparing two methods of computing the ranks of the (isomorphic) vector spaces occurring in Theorem \ref{main-result}, using the work of Galatius and Randal-Williams on the left-hand side (which has been described in detail in Section \ref{sec:MTWTheory}), and using the work of Berglund and Madsen as well as classical invariant theory on the right-hand side.

\subsection*{Acknowledgements}

We are grateful to a number of colleagues for useful conversation, namely Holger Kammeyer (on Lie algebras), Wolfgang L\"uck and in particular Michael Weiss. O.\ Randal-Williams was supported by the Herchel Smith Fund.

\section{Diffeomorphism groups, Thom spectra and index theory}\label{section:gmtw-theory-index-theory}

One of the key results of this paper is the following theorem.

\begin{theorem}\label{mt-spectra-index}
There exists a homotopy-commutative diagram 
\begin{equation}\label{the-main-diagram}
\begin{gathered}
\xymatrix{
B \Diff (W_{g}^{2n},D) \ar[d]^{\alpha_{g}^{2n}} \ar[r]^-{\zeta} & B \Gamma(W_{g}^{2n}) \ar[d]^{\beta_{g}^{2n}}\\
\loopinf{}_0 \MTthetan \ar[r]^-{\symb} & \loopinf{+k_n}_0 \KO.
}
\end{gathered}
\end{equation}
The number $k_n$ equals $0$ if $n$ is even and $2$ if $n$ is odd.
The left vertical map is an integral homology equivalence in degrees $* \leq (g-3)/2$. The bottom horizontal map is injective in rational cohomology and its image (in positive degrees) is the subalgebra generated by the $\kappa_{\cl_i}$. The right vertical map is a rational homology equivalence in degrees $* \leq g-2$.
\end{theorem}

In this section, we will define all remaining spaces and maps and prove most of the theorem (this is mostly a recollection of known results); in the next section we will explain how to derive from Borel's work that the right vertical map is a rational homology equivalence in the stated degrees.

\subsection{The action on the middle homology}\label{sec:DefiningGamma}

Let $W^{2n}_{g}:= \sharp^{g} (S^n \times S^n)$ be the connected sum of $g$ copies of $S^n \times S^n$, $D \subset W_{g}^{2n}$ be a fixed closed disc, and $W^{2n}_{g,1} := W^{2n}_g \setminus \mathrm{int}(D)$. When the dimension $2n$ is understood, we write $W_{g,1} := W_{g,1}^{2n}$. We assume throughout that $n \geq 3$.

\begin{definition}
The \emph{Torelli group} $\Tor_{g,1}^{2n} \subset \Diff_\partial (W_{g,1}^{2n})$ is the subgroup of all diffeomorphisms that act as the identity on $H_n (W_{g,1}^{2n};\bZ)\cong \bZ^{2g}$.
\end{definition}

The middle-dimensional homology $H_n (W_{g,1};\bZ)\cong H_n (W_g;\bZ)\cong \bZ^{2g}$ carries a wealth of algebraic structure. The \emph{homological intersection pairing} is a $(-1)^n$-symmetric bilinear form $I: H_n (W_{g};\bZ) \otimes H_n(W_{g};\bZ) \to \bZ$, which is non-degenerate by Poincar{\'e} duality. Let us define
\begin{equation*}
\Lambda_n = \begin{cases}
0 & \text{$n$ even}\\
\bZ & \text{$n=1$, $3$, or $7$}\\
2\bZ & \text{else}.
\end{cases}
\end{equation*}
The \emph{quadratic refinement} is a map $q: H_n (W_{g};\bZ) \to \bZ/\Lambda_n$ that satisfies $q(x+y)=q(x)+q(y) + I(x,y) \,\,\mathrm{mod}\,\, \Lambda_n$ and $q(\lambda \cdot x) = \lambda^2 \cdot q(x)$. It was introduced by Wall \cite[p.\ 167f]{Wall1}, and we give a sketch of the definition. By a theorem of Haefliger \cite{Haefliger}, an element $x \in H_n (W^{2n}_g;\bZ)$ may be represented by an embedded sphere as long as $n \geq 3$, and it is unique up to isotopy as long as $n \geq 4$.  The normal bundle of this embedding can be viewed as an element of $\pi_n (B\Ort(n))$, and as $W_g^{2n}$ is stably parallelisable this element lies in the kernel of the stabilisation map $\pi_n (B\Ort(n))\to \pi_n (B\Ort)$, which is canonically identified with $\bZ/\Lambda_n$ (generated by the tangent bundle of $S^n$). Let $q(x) \in \bZ/\Lambda_n$ be the corresponding element (note that the non-uniqueness of an embedded representative of $x$ when $n=3$ does not matter, as $q$ takes values in the trivial group in that case).

Let $\Gamma(W_{g}) $ be the group of $\bZ$-linear automorphisms of $H_n (W_{g};\bZ)$ which preserve the intersection form and the quadratic refinement. There is a basis of $H_n (W_{g};\bZ)$, say $(x_1, \ldots, x_g,y_1,\ldots , y_g)$, such that $I(x_i,x_j)=I(y_i,y_j)=0$, $I(x_i,y_j)=\delta_{i,j}$, and $q(x_i)=q(y_j)=0$. When $n$ is even the quadratic property implies that $2q(x) = I(x,x)$, so $q$ contains no information beyond $I$. When $n$ is 1, 3, or 7 then $q$ takes values in the trivial group so contains no information, but for other odd $n$ the form $q$ contains information which cannot be recovered from $I$. The automorphism groups of these data are thus
\begin{equation*}
\Gamma(W_{g}) = \begin{cases}
\Ort_{g,g}(\bZ) & \text{$n$ even}\\
\Sp_{2g}(\bZ) & \text{$n=1$, $3$, or $7$}\\
\Gamma_g(1,2) & \text{else}.
\end{cases}
\end{equation*}
To explain the notation, let us write $J_{\pm ,g}:= \bigl(\begin{smallmatrix}
0&I\\ \pm I&0 \end{smallmatrix} \bigr)$ and then define, for each commutative ring $R$,
\begin{equation*}
\begin{aligned}
\Ort_{g,g}(R) &:= \{ A \in \GL_{2g} (R) \,\vert\, A^T J_{+,g} A = J_{+,g} \}\\ \Sp_{2g}(R) &:= \{ A \in \GL_{2g} (R) \,\vert\, A^T J_{-,g} A = J_{-,g} \}.
\end{aligned}
\end{equation*}
Finally, $\Gamma_g(1,2)$ is the finite index subgroup of $\Sp_{2g}(\bZ)$ of elements which preserves the quadratic refinement $q$ (this group is commonly considered in the theory of theta functions, from which we borrow the notation $\Gamma_g(1,2)$). 

Each orientation-preserving diffeomorphism of $W_g$ induces an automorphism of $H_n (W_g;\bZ)$ which preserves the intersection form and the quadratic refinement. Thus we obtain a group homomorphism $\Diff^+ (W_g) \to \Gamma (W_g)$.

\begin{proposition}\label{torelli-space--connected}
The composition $\Diff (W_g,D) \to \Diff^+ (W_g) \to \Gamma(W_g)$ is surjective. Thus the classifying space $B \Tor_{g,1}^{2n}$ is weakly equivalent to the homotopy fibre of the induced map $B \Diff (W^{2n}_g,D) \to B \Gamma (W^{2n}_g)$.
\end{proposition}
\begin{proof}
Kreck \cite{Kreck} has shown that $\pi_0(\Diff^+ (W_g)) \to \Gamma(W_g)$  is surjective, building on work of Wall \cite{Wall2}. 
Therefore in order to show the first part it is enough to show that $\pi_0 (\Diff (W_g,D) ) \to \pi_0 (\Diff^+ (W_g))$ is surjective. Let $\Emb^+ (D,W_g)$ be the space of orientation-preserving embeddings. By the ``disc lemma'' \cite[Lemma 10.3]{BJ}, $\Emb^+ (D,W_g)$ is connected; but the restriction map $ \Diff^+ (W_g) \to \Emb^+ (D, W_{g}) $ is a principal bundle with group $\Diff (W_g,D)$, therefore the result follows by the long exact homotopy sequence. The second part follows from the first by standard bundle techniques.
\end{proof}

\subsection{Madsen--Tillmann--Weiss theory}\label{sec:MTWTheory}

Let $E \to B$ be a fibre bundle with structural group $\Diff (W_g,D)$ and fibre $W_g$, a structure which we call a ``smooth $(W_g,D)$-bundle''. Such a bundle comes with a bundle map $d: B \times D^{2n} \to E$ which is a fibrewise smooth embedding, and a vertical tangent bundle $T_v E$ with classifying map $\tau_v E : E \to B\Ort(2n)$. The bundle $d^* T_v E$ is identified with the vertical tangent bundle of $B \times D^{2n} \to B$, which is canonically trivialised (using the framing of $D^{2n} \subset \bR^{2n}$). We write $E_0$ for the image of $d$, considered as a subundle of $E$.

Let $\theta^n: B\Ort(2n) \langle n \rangle \to B\Ort(2n)$ be the $n$-connective cover (by which we mean that $\pi_i(B\Ort(2n) \langle n \rangle)=0$ for $i \leq n$ and $\theta^n $ induces an isomorphism on homotopy groups of higher degrees; alternatively the homotopy fibre of $\theta^n $ is $n$-co-connected). As $T_v E\vert_{E_0}$ is canonically trivialised, the restriction $\tau_v E\vert_{E_0}$ has a canonical lift along the fibration $\theta^n$. An easy application of obstruction theory, using that $H^i (E, E_0;\bZ)=0$ for $i \leq n-1$, proves that the given lift of $\tau_v E|_{E_0}$ through $\theta^n$ extends uniquely (up to homotopy) to $E$. 
In the language of \cite{GMTW}, this means that each smooth $(W_g, D)$-bundle has a canonical $\theta^n$-structure. Let $\MTthetan$ denote the Madsen--Tillmann--Weiss spectrum for the tangential structure $\theta^n$ (this is the Thom spectrum of $-(\theta^n)^* \gamma_{2n}$, where $\gamma_{2n} \to B\Ort(2n)$ is the universal vector bundle). There is a map (arising from the Pontrjagin--Thom construction) $\alpha_E: B \to \loopinf{}_{0} \MTthetan$ to the unit component of the infinite loop space. In the universal case, we obtain a map
\begin{equation}\label{grw-map}
\alpha_g: B \Diff (W^{2n}_g,D)\lra \loopinf{}_{0} \MTthetan.
\end{equation}

\begin{theorem}[Galatius, Randal-Williams \cite{GRW1,GRW2}]
The map (\ref{grw-map}) induces an isomorphism in integral homology in degrees $* \leq (g-3)/2$.
\end{theorem}

The map $\alpha_E$ has a close connection to the Miller--Morita--Mumford classes. If $p:E \to B$ is an oriented smooth manifold bundle with fibre dimension $2n$ and $c \in H^{k+2n} (B\SO(2n))$ a characteristic class, we obtain a class
$$\kappa_c (E):=p_{!} (c(T_v E)) \in H^{k} (B),$$
called the generalised Miller--Morita--Mumford class corresponding to $c$. There is a map of graded vector spaces
\begin{equation*}\label{cohomology-mttheta2}
H^{*} (B\Ort(2n)\langle n \rangle; \bQ)[-2n] \cong H^*(\MTthetan;\bQ) \overset{\sigma}\lra H^* (\loopinf{}_{0} \MTthetan;\bQ)
\end{equation*}
where the first map is the Thom isomorphism and the second is the cohomology suspension, and we denote the image of a class $c \in H^{k+2n}(B\Ort(2n)\langle n \rangle; \bQ)$ by $\kappa_c \in H^k (\loopinf{}_{0} \MTthetan;\bQ)$. For any smooth $(W_g, D)$-bundle $p : E \to B$ and class $c$ of degree $k+2n > 2n$ we obtain the equation
$$\alpha_E^*(\kappa_c) = \kappa_c(E) \in H^k(B;\bQ).$$

Let us now describe the rational cohomology of $\loopinf{}_{0} \MTthetan$. Recall that the rational cohomology ring of $B\SO(2n)$ has a presentation
$$H^* (B\SO(2n); \bQ) = \bQ[p_1,\ldots,p_n,e]/(e^2 -p_n),$$
where $e \in H^{2n}(B\SO(2n); \bQ)$ is the Euler class and $p_i \in H^{4i}(B\SO(2n); \bQ)$ is the $i$th Pontrjagin class. For our purposes, a different system of algebra generators is more convenient. Let $\cl_i \in H^{4i} (B\SO(2n); \bQ)$ be the $i$th component of the Hirzebruch $\mathcal{L}$-class. Then each $p_i$ can be written as a polynomial in the classes $\cl_1, \ldots, \cl_i$ (this follows from the fact that the coefficient of $p_n$ in $\cl_n$ is nonzero, see \cite[p.\ 230]{MS}), as $p_i = p_i(\cl_1, \ldots, \cl_i)$. Hence we may equally well describe the rational cohomology of $B\SO(2n)$ as
$$H^* (B\SO(2n); \bQ) = \bQ[\cl_1,\ldots,\cl_n,e]/(e^2 -p_n(\cl_1, \ldots, \cl_n)).$$
The cohomology of the $n$-connected cover $B\SO(2n)\langle n \rangle$ of $B\SO(2n)$ may then be described as
\begin{equation*}
H^* (B\SO(2n)\langle n \rangle; \bQ) = \bQ\left[\cl_{\ceil{\frac{n+1}{4}}}, \ldots, \cl_n, e \right]/(e^2 -p_n(0, \ldots, 0, \cl_{\ceil{\frac{n+1}{4}}}, \ldots, \cl_n)).
\end{equation*}
Let $I :=\{ \ceil{\frac{n+1}{4}},\ldots ,n \}$, and for a multiindex $\B{i} \in {\bN_{0}}^{I}$ write $|\B{i}|: =\sum_{j \in I} i_j$ and $w (\B{i}):= 4 \sum_{j \in I} i_j j$. Define elements $\lambda_{\B{i}}$ and $\mu_{\B{i}}$ in $ H^* (\loopinf{}_{0} \MTthetan;\bQ)$ by
\begin{equation*}\label{def-lambda-classes}
\lambda_{\B{i}}:=\kappa_{\prod_{j \in I} \cl_j^{i_j}} \in H^{w (\B{i})-2n} (\loopinf{}_{0} \MTthetan;\bQ) 
\end{equation*}
and 
\begin{equation*}\label{def-mu-classes}
\mu_{\B{i}}:=  \kappa_{e \cdot \prod_{j \in I} \cl_j^{i_j}} \in H^{w (\B{i})} (\loopinf{}_{0} \MTthetan;\bQ).
\end{equation*}
Then the natural map
\begin{equation}\label{eq:cohom-calc}
\bQ[\lambda_{\B{i}},\mu_{\B{j}} \,\,\vert\,\, \B{i},\B{j} \in \bN_{0}^{I}, w (\B{i})>2n, w(\B{j}) >0] \lra H^{*} (\loopinf{}_{0} \MTthetan;\bQ). 
\end{equation}
is an isomorphism of algebras, by \cite[\S 2.5]{GRW2}.

\subsection{Some maps of classifying spaces}

In this subsection, we will define the right vertical maps of diagram \eqref{the-main-diagram}. These are only defined up to homotopy, and they are the compositions
\begin{equation*}
\beta_{g}^{2n}:B \Ort_{g,g}(\bZ) \stackrel{\psi}{\lra} B\Ort_{g,g} (\bR) \stackrel{\mu}{\lra} B\Ort(g) \times B\Ort(g) \stackrel{\eta}{\lra} B\Ort = \loopinf{}_{0} \KO 
\end{equation*}
if $n$ is even and 
\begin{equation*}
\beta_{g}^{2n}:B \Gamma (W^{4n+2}_{g}) \stackrel{\psi}{\lra} B \Sp_{2g}(\bR) \stackrel{\mu}{\lra} B\U(g) \stackrel{\eta}{\lra} \SO /\U = \loopinf{+2}_{0} \KO
\end{equation*}
if $n$ is odd. The maps called $\psi$ are induced by the inclusions of groups. Observe that $\Ort(g) \times \Ort(g) \subset \Ort_{g,g}(\bR)$ and $\U(g) \subset \Sp_{2g}(\bR)$ are maximal compact subgroups and therefore the inclusions induce homotopy equivalences on classifying spaces; the maps $\mu$ are by definition homotopy inverses of these maps. For the last maps, let $\eta: B\Ort(g) \times B\Ort(g) \to B\Ort$ be the difference with respect to the Whitney sum. In the odd case, let $W \to B\U(g)$ be the universal bundle and consider the map $\Delta: B\U(g) \to B\U$ which classifies the virtual bundle $[W] - [\overline{W}]$. There is a nullhomotopy of the composition $B\U(g) \to B\U \to B\SO$, since $[W] - [\overline{W}]$ is canonically trivial as a real virtual vector bundle. So we get a map $\eta: B\U(g) \to \SO/\U$ into the homotopy fibre of the realification map $B\U\to B\SO$.

Let us explain the map $ \mu$ in a little more detail, starting with the orthogonal case. Let $V \to B \Ort_{g,g}(\bR)$ be the universal vector bundle, which comes with a fibrewise symmetric nondegenerate bilinear form $J$. The space $B\Ort(g) \times B\Ort(g)$ can be described as the total space of the fibre bundle associated to $E \Ort_{g,g} (\bR) \to B \Ort_{g,g} (\bR)$ with fibre the (contractible) homogeneous space 
$$\frac{\Ort_{g,g} (\bR)}{\Ort(g) \times \Ort(g)} \cong \{ \sigma \in \GL_{2g}(\bR) \,\vert\, \sigma^2 =1 ; \; \sigma^T J_{+,g} = J_{+,g} \sigma ;  \; J \sigma >0\}.$$
Thus we can describe the map $ \mu$ concretely by the following procedure. Choose an involution $\sigma$ on the universal vector bundle $V \to B\Ort_{g,g} (\bR)$, such that $J(v, \sigma w) = J (\sigma v, w)$ and $J(v, \sigma v)>0$ holds for all $v,w$. This is possible by the contractibility of the homogeneous space. Let $V_{\pm}$ be the $(\pm 1)$-eigenbundle of $\sigma$. The map $\mu$ is a classifying map for the pair $(V_+,V_-)$ of vector bundles on $B \Ort_{g,g} (\bR)$ and $\eta \circ \mu:B\Ort_{g,g} (\bR) \to B\Ort$ represents the $KO$-theory class $[V_+] - [V_-]$.

In the symplectic case, the contractible homogeneous space is 
$$\frac{\Sp_{2g}(\bR)}{\U(g)} \cong \{ \sigma \in \GL_{2g}(\bR) \,\vert\, \sigma^2 =-1 ; \; \sigma^T J_{-,g} =- J_{-,g} \sigma ;  \; J \sigma >0\}.$$
In symplectic linear algebra, this is known as the space of compatible complex structures to the symplectic form $J$. On the universal bundle $V \to B \Sp_{2g}(\bR)$, there is a skew-symmetric form $J$. Pick a compatible complex structure $\sigma$ on $V$ (i.e. $\sigma^2=-1$, $J(\sigma v,w)=-J(v,\sigma w)$ and $J(v, \sigma v )>0$). Then $(V, \sigma)$ is a complex vector bundle ($\sigma$ is multiplication by $i$). The classifying map of this vector bundle gives the homotopy equivalence $\mu: B \Sp_{2g} (\bR) \to B\U(g)$. 

\begin{remark}\label{remark-beta-odd}\mbox{}
\begin{enumerate}[(i)]
\item The composition $B\U(g) \stackrel{\eta}{\to} \SO/\U \to B\U$ of $\eta$ with the classifying map of the $\U$-bundle $\SO \to \SO/\U$ is $\Delta$.
\item It is useful to have a slightly different description of $\eta \circ \mu$ in the symplectic case: pick a compatible complex structure $\sigma$ on the universal bundle $V \to B \Sp_{2g}(\bR)$, and consider the virtual complex bundle $[\Eig (\sigma \otimes \bC,i)] - [\Eig (\sigma \otimes \bC,-i)]$, which is canonically trivial as a real virtual bundle.
\end{enumerate}
\end{remark}

\subsection{Connection to index theory}

We will now describe the composition 
\begin{equation*}
B \Diff (W_{g}^{2n}, D) \lra B \Gamma(W_{g}^{2n}) \stackrel{\beta_{g}^{2n}}{\lra} \loopinf{+k_n}_{0} \KO 
\end{equation*}
in index-theoretic terms. The Atiyah--Singer family index theorem will then prove the homotopy-commutativity of the diagram (\ref{the-main-diagram}) (and in particular, provide the last map). Let us recall the signature operator, following \cite[\S 6]{ASIII}. 

Let $M^{2n}$ be a closed oriented Riemannian manifold. Let $\bK$ be either $\bR$ or $\bC$ and $\cA^p (M;\bK)$ be the space of $\bK$-valued $p$-forms. 
Let $\star: \cA^p (M;\bR)\to \cA^{2n-p}(M;\bR)$ be the Hodge star operator, which is an operator of order $0$. An inner product on $\cA^* (M;\bR)$ is given by $\langle \omega, \eta\rangle  := \int_M  \omega \wedge \star \eta$. 
We extend the Hodge star to a $\bC$-linear operator on $\cA^* (M; \bC)$ and let $\kappa: \cA^p (M;\bC)\to \cA^p(M;\bC)$ be the complex conjugation $\omega \mapsto \bar{\omega}$. The inner product is extended in the standard way to $\cA^* (M; \bC)$.
Consider $D=d+d^*:\cA^* (M;\bK)\to \cA^* (M;\bK)$, a formally self-adjoint elliptic differential operator. On the Sobolev completion of $\cA^* (M;\bK)$, $D$ induces a Fredholm operator, also denoted $D$. The kernel of $D$ is the space $\Harm^* (M;\bK)=\bigoplus_{p=0}^{2n} \Harm^p (M;\bK)$ of $\bK$-valued harmonic forms on $M$, which by the Hodge theorem can be identified with $H^* (M; \bK)$. 
The \emph{cohomological intersection pairing} on $\Harm^* (M; \bK)$ is given by $J(\omega,\eta)=\int_M \omega \wedge \eta$. Note the relation (for $\bR$-valued forms)
\begin{equation}\label{cohomological-intersection-inner-product}
\langle \omega, \eta \rangle = J (\omega, \star \eta).
\end{equation}

Now define $\tau: \cA^p (M;\bC)\to \cA^{2n-p}(M;\bC)$ by $\tau: =i^{p(p-1)+n}\star$. The operators $D$ and $\tau$ are $\bC$-linear, while $\kappa$ is $\bC$-antilinear and there are the following relations
\begin{equation}\label{relations-d-tau-kappa}
\tau^2 = 1; \; \kappa^2 = 1; \; \kappa \tau=(-1)^n\tau \kappa; \; D \kappa=\kappa D; \; D \tau = - \tau D 
\end{equation}
between these maps (the third equation says that $\tau$ is real if $n$ is even and imaginary if $n$ is odd, the fourth equation says that $D$ is real). Denote by $\cA^{\pm}(M;\bC)$ be $(\pm 1)$-eigenspace of $\tau$ (as $\tau$ is tensorial, this is indeed the space of sections of a vector bundle).
The operator $D$ restricts to $D^{\pm}: \cA^{*} (M;\bC)^{\pm} \to \cA^{*} (M;\bC)^{\mp}$, and the operators $D^+$ and $D^-$ are mutually adjoint Fredholm operators. The kernel of $D^{\pm}$ is the $(\pm 1)$-eigenspace of the action of $\tau$ on $\Harm^*(M;\bC)$.
We consider the index $\ind (D)=[\ker (D^+)] - [\ker (D^{-})]\in K^0 (*)=\bZ$. 
If $n$ is odd, then $\kappa$ gives a $\bC$-antilinear isomorphism 
\begin{equation*}
\ker (D^+) \to \ker (D^-), 
\end{equation*}
and hence $\ind(D)=0$. If $n$ is even, $\ind(D) = \sign(M)$ is the signature of $M$. 

These structures admit a generalisation to fibre bundles. Let $\pi:E \to B$ be a smooth oriented $M$-bundle (on a paracompact base space). The fibres are denoted $E_b=\pi^{-1}(b)$. Pick a fibrewise Riemannian metric on $E$: the operators $D_b$ on the manifolds $E_b$ assemble to a family $\{D_b\}_{b}$ of elliptic operators, which gives a family of Fredholm operators (between suitable Hilbert bundles over $B$). 
Because the dimension of the kernel of $D_b$ does not depend on $b$, the union $\Harm^* (\pi;\bK):=\cup_{b \in B} \Harm^* (E_b;\bK) \to B$ is a finite-dimensional $\bK$-vector bundle (it is isomorphic to the bundle over $B$ whose fibre over $b$ is the $\bK$-cohomology of $E_b$). The operators $\tau_b$ and $\kappa_b$ on $\Harm^* (E_b; \bC)$ yield bundle maps $\tau$ and $\kappa$ of $\Harm^* (\pi; \bC)$, satisfying the relations \eqref{relations-d-tau-kappa}. The $(\pm 1)$-eigenbundles of $\tau$ are denoted $\Harm^{\pm} (\pi;\bC)$, and $D$ . The family index of the signature operator $D$ is the formal difference 
$$\ind_{\bC}(D):=[\Harm^+(\pi;\bC)] - [\Harm^- (\pi;\bC)] \in K^0 (B).$$
Depending on the parity of $n$, there is a refinement of this index to an index in real $K$-theory.

If $n=2m$ is even, the operator $\tau$ is real and $\Harm^{\pm} (\pi; \bC) = \Harm^{\pm}(\pi; \bR)\otimes \bC$. Thus the index $\ind (D)$ is the image of the element 
$$\ind_{\bR}(D):=[\Harm^{+}(\pi; \bR)]-[\Harm^{-}(\pi; \bR)] \in KO^0 (B)$$
under the complexification map $KO^0 (B) \to K^0 (B)$. If $M= W^{4m}_{g}$, we obtain a map 
\begin{equation}\label{universalindex-evencase}
\ind_{\bR} (D):B \Diff(W^{4m}_{g},D^{4m}) \to \{0\} \times B\Ort \subset \bZ \times B\Ort,
\end{equation}
which lands in the zero component since the signature of $W^{4m}_{g}$ is zero.

\begin{proposition}
The map $\ind_{\bR} (D)$ in \eqref{universalindex-evencase} is homotopic to the composition 
$$B \Diff (W_{g}^{4m}) \stackrel{\zeta}{\lra} B\Ort_{g,g} (\bZ) \stackrel{\beta^{4m}_{g}}{\lra} B\Ort.$$
\end{proposition}
\begin{proof}
Pick a fibrewise Riemannian metric on the universal $W_{g}^{4m}$-bundle $\pi$ over $B \Diff (W_{g}^{4m})$. The map $\psi \circ \zeta : B \Diff (W_{g}^{4m}) \lra B\Ort_{g,g} (\bZ) \to B \Ort_{g,g} (\bR)$ is a classifying map for the bundle $\Harm^{2m} (\pi; \bR)$, equipped with the bilinear form $J$ given by the cohomological intersection pairing. Note that $\star$ is an involution, $J(-, \star -)$ is a scalar product by \ref{cohomological-intersection-inner-product} and 
$$J (\star \omega , \eta) = J (\eta, \star \omega )=\langle \eta, \omega \rangle = \langle \omega, \eta  \rangle = J(\omega, \star \eta).$$
By the recipe given in the last section, the composition $\beta^{4m}_{g}\circ \zeta$
represents the $KO$-theory class $[\Eig (\star,+1)] - [\Eig (\star,-1)]=  [\Harm^{+,2m} (\pi; \bR) ] -[\Harm^{-,2m} (\pi; \bR) ]$, the last equation holds since in the middle dimension $\tau=\star $. 

Since $W^{4m}_{g}$ is $(2m-1)$-connected, harmonic forms only exist in degrees $0, 2m$ and $4m$. The contribution from $0$ and $4m$-forms to the index is zero, and this finishes the proof.
\end{proof}

Now we turn to the case $n = 2m+1$. In that case, $\kappa $ defines a $\bC$-antilinear isomorphism $\ker (D^+) \to \ker (D^-)$. Therefore $\ind_{\bC} (D)$ maps to $0$ under the realification map $K^0 \to KO^0$. Therefore we get a map 
\begin{equation}\label{universalindex-oddcase}
\ind_{\bR} (D):B \Diff(W^{4m}_{g},D^{4m+2}) \lra \SO/\U
\end{equation}
into (the connected component of) the homotopy fibre of $B\U \to B\SO$, and the composition of $\ind_{\bR}(D)$ with $\SO/\U \to B\U$ gives $\ind(D)_{\bC}$.

\begin{proposition}
The map $-\ind_{\bR} (D)$ in \eqref{universalindex-oddcase} is homotopic to the composition 
\[
B \Diff (W_{g}^{4m+2}) \stackrel{\zeta}{\lra} B\Gamma (W^{4m+2}_{g} )  \stackrel{\beta^{4m}_{g}}{\lra} \SO/\U.
\]
\end{proposition}

\begin{proof}
This is analogous to the previous proof. Pick a fibrewise Riemannian metric on the universal $W_{g}^{4m+2}$-bundle $\pi$ over $B \Diff (W_{g}^{4m+2})$. The map $B \Diff (W_{g}^{4m+2}) \stackrel{\psi\circ \zeta}{\lra}  B \Sp_{2g} (\bR)$ is a classifying map for the bundle $\Harm^{2m+1} (\pi; \bR)$, equipped with the bilinear form $J$ given by the cohomological intersection pairing. Note that $\star^2 =-1$ and $\tau=i \star$ in the middle dimension, $J(-, \star -)$ is a scalar product by \ref{cohomological-intersection-inner-product} and $J(\omega,\star \eta) = - J(\star \omega,\eta)$ and so $\star$ is a compatible complex structure. So according to the recipe from Remark \ref{remark-beta-odd}, we have to consider the virtual bundle $[\Eig (\star,+i)] - [\Eig (\star,-i)] = [\Eig (\tau,-1)] - [\Eig (\tau,+1)]$, which gives $- \ind_{\bR}(D)$ (when mapped to the homotopy fibre $\SO/\U$).
Since $W^{4m+2}_{g}$ is $2m$-connected, harmonic forms only exist in degrees $0, 2m+1$ and $4m+2$. The contribution from $0$ and $4m+2$-forms to the index is zero, and this finishes the proof.
\end{proof}

\begin{remark}
The choice of the targets of the index maps look a bit artificial. The canonical expectation is that there is an index map $B \Diff^+ (M^{2n}) \to \loopinf{+2n} \KO$ for all $n$. The problem is that the signature operator does not define a $KO$-orientation of any oriented manifold. The whole problem is a $2$-primary problem, since after inverting $2$, the $KO$-spectrum becomes $4$-periodic. As the present paper is ultimately only about rational invariants, we have decided to ignore this issue and stick to the artificial construction.
\end{remark}

The signature operator is a universal operator on oriented manifolds in the sense of \cite{Eb}. Thus an exercise with the Atiyah--Singer family index theorem (which has been solved in loc.\ cit.) shows that the index of the signature operator factors through the Madsen--Tillmann--Weiss spectrum. More precisely, write $k_n=0$ if $n$ is even and $k_n=2$ if $n$ is odd. Then there exists an infinite loop map $\symb:\loopinf{} \MTSO (2n)\to \loopinf{+k_n}\KO$, such that the composition
\begin{equation}\label{indexsequence}
B \Diff^+ (M) \stackrel{MTW}{\lra} \loopinf{} \MTSO(2n) \stackrel{\symb}{\lra} \loopinf{+k_n} \KO
\end{equation}
is homotopic to the index of the signature operator (the map in \ref{indexsequence} is the topological index). Abusing notation slightly, we denote the composition of $\symb$ with the natural map $\loopinf{}\MTthetan \to \loopinf{}\MTSO(2n)$ by $\symb$ too, and we get that
$$
B \Diff (W_{g},D) \stackrel{\eqref{grw-map}}{\lra} \loopinf{}_{0} \MTthetan  \stackrel{\symb}{\lra} \loopinf{+k_n}_{0} \KO$$
is homotopic to the the index of the signature operator, i.e. to $\ind(D)$. This finishes the construction of the diagram (\ref{the-main-diagram}).

The effect of the map $\ind(D)$ in rational cohomology is well-understood. Recall that there is a slight incompatibility of Hirzebruch's original classes with those showing up in index theory. More specifically, Hirzebruch considered the class $\cl$ which is associated with the formal power series of $\frac{x}{\tanh(x)}$ through the formalism of multiplicative sequences. In index theory, one uses the class $\widehat{\cl}$ associated with $\frac{x/2}{\tanh(x/2)}$. It is clear that $\cl_i = 2^{2i} \cdot \widehat{\cl}_i$.

If $n=2m$ is even, consider the complexification map $\loopinf{} \KO \to \loopinf{} \KU$. The image of the Chern character class $\ch_{2i+1} \in H^{2i+2} (B\U)$ in $H^* (B\Ort;\bQ)$ is zero and $(-1)^i \ch_{2i}$ maps (by definition) to the Pontrjagin character $\ph_i$. It is well-known that the cohomology of $\loopinf{}_{0} \KO \simeq B\Ort$ is the polynomial algebra $\bQ[\ph_1, \ldots ]$. Moreover, $\ind(D)^*(\ph_i) = (-\tfrac{1}{4})^i \kappa_{\cl_{i+m}}$. This computation can be found in \cite[Corollary III.15.4]{LM}  (see also p.\ 232 loc.\ cit.). Thus in positive degrees the map $ \symb^*: H^* (\loopinf{}_{0} \KO; \bQ) \to H^{*} (\loopinf{}_{0} \MTthetan; \bQ)$ is injective and its image is the subalgebra generated by the $\kappa_{\cl_i}$.

If $n=2m+1$ is odd, the composition $B \Diff^+ (W_{g}) \to \SO /\U \to B\U$ is the complex family index and the same computation yields that it pulls back $\ch_{2i} $ to zero and $\ch_{2i-1}$ to $(\tfrac{1}{2})^{2i-1}\kappa_{\cl_{i+m}}$. On the other hand, let $\qh_j \in H^{2j-2}(\SO /\U)$ be the pullback of $\ch_{2j-1}$ under $\SO /\U \to B\U$. It is well-known that $\bQ[\qh_1,\qh_3, \ldots ] \cong H^* (\SO/\U;\bQ)$. In other words, the map $ \symb^*: H^* (\loopinf{+2}_{0} \KO; \bQ) \to H^{*} (\loopinf{}_{0} \MTthetan; \bQ)$ is injective and (in positive degrees) its image is the subalgebra generated by the $\kappa_{\cl_i}$.

\section{Borel's result on the cohomology of arithmetic groups}\label{section:arithmetic-groups}

In this section, we discuss those aspects of A. Borel's work on the cohomology of arithmetic groups that are relevant for us. 
This serves a twofold purpose. The first purpose is to show that the right  vertical map in Theorem \ref{mt-spectra-index} is a rational cohomology equivalence in a stable range. This is accomplished by Theorem \ref{cohomology-arithmetic-group-constant-coeffs} below. The second goal is to discuss a less well-known result stating that $H^* (\Gamma(W_g);V)=0$ in a range of degrees for certain rational $\Gamma(W_g)$-modules $V$.

\begin{theorem}[Borel]\label{cohomology-arithmetic-group-constant-coeffs}\mbox{}
\begin{enumerate}[(i)]
\item $(\beta_{g}^{4m})^{*}:H^q (B\Ort; \bQ) \to H^q (B \Ort_{g,g}(\bZ); \bQ)$ is an isomorphism for $q \leq g-2 $.
\item $(\beta_{g}^{4m+2})^{*}:H^q (\SO/\U; \bQ) \to H^q (B\Gamma(W_g^{4m+2}); \bQ)$ is an isomorphism for $q \leq g-1$.
\end{enumerate}
\end{theorem}

Theorem \ref{cohomology-arithmetic-group-constant-coeffs} is a reformulation of the main result of the classical paper \cite{Borel1}. Later, Borel proved an improved range of degrees (as stated above) and an extension to nontrivial coefficients which we will also use. The extended result is announced in \cite{Borel2} and proved in \cite{Borel3}.
We discuss how to derive Theorem \ref{cohomology-arithmetic-group-constant-coeffs} because we are not aware that it is published as stated.

\subsection{The Borel--Matsushima homomorphism}

Let $G$ be a real semisimple Lie group with finitely many components and finite centre, $K \subset G$ be a maximal compact subgroup, and $\Gamma \subset G$ a discrete subgroup. Let $G_{\bC}$ be a complexification of $G$ and $L \subset G_{\bC}$ be a maximal compact subgroup that contains $K$. The cases relevant to us will be as follows.
\begin{center}
\begin{tabular}{ccccc}
$G$ & $K$ & $\Gamma$ & $G_\bC$ & $L$ \\
\hline
$\Ort_{g,g}(\bR)$ & $\Ort(g) \times \Ort(g)$ & $\Ort_{g,g}(\bZ)$ & $\Ort_{g,g}(\bC)$ & $\Ort(2g)$   \\
$\SO_{g,g}(\bR)$ & $\mathrm{S}(\Ort(g) \times \Ort(g))$ & $\SO_{g,g}(\bZ)$ & $\SO_{g,g}(\bC)$ & $\SO(2g)$   \\
$\Sp_{2g}(\bR)$ & $\U(g)$ & $\Sp_{2g}(\bZ)$ & $\Sp_{2g}(\bC)$ & $\USp(g)$    \\
$\Sp_{2g}(\bR)$ & $\U(g)$ & $\Gamma_g(1,2)$ & $\Sp_{2g}(\bC)$ & $\USp(g)$    \\
\end{tabular}
\end{center}
Here, we denote by $\USp (g)$ the group of $\bH$-linear isometries of $\bH^g$ (this group is often denoted $\Sp (g)$ in compact Lie group or topology texts). Let $\fg$, $\fk$, and $\fl$ be the Lie algebras of $G$, $K$, and $L$ respectively and let $G^0 \subset G$ and $K^0 = K \cap G^0 \subset K$ be the identity components. Let $X=G/K=G^0 / K^0$ be the symmetric space associated with $G$, which is contractible and has a proper action of $\Gamma$, so the quotient $\Gamma \backslash X$ has the rational (or complex) cohomology of $B\Gamma$. Let $I_G \subset \cA^* (X)$ be the subspace of complex-valued differential forms on $X$ invariant under the action of $G^0$. 
It is a well-known result, see e.g. \cite[Prop. 7.4.14]{Nic}, that invariant forms on a symmetric space are closed, so all elements of $I_G$ are closed forms. The group $G$ (and hence $\Gamma$) acts on $I_G$ via its finite quotient $G/G^0$, and there is a canonical homomorphism 
\begin{equation}\label{borelmatsushima}
j^*:I_{G}^{\Gamma} \lra H^*(\Gamma \backslash X; \bR) \cong H^* (B\Gamma;\bR).
\end{equation}

Borel \cite{Borel1} proved that $j^*$ is an isomorphism in small degrees when $G$ is the group of real points of a semisimple algebraic group defined over $\bZ$ and $\Gamma$ is an arithmetic subgroup. We will state these results further below in more detail, but before, we would like to discuss another description of $j^*$ which has a more homotopical flavour (and does not involve algebraic groups), the goal is to prove Proposition \ref{borel-mathsushima3} below.

Let $C^{*}(\fg, \fk) \subset \Lambda^* \fg^*$ be the subcomplex of the Chevalley--Eilenberg complex consisting of the forms $\omega$ with $\iota_x \omega =0$ and $L_x \omega=0$ for all $x \in \fk$, where $\iota_x$is the insertion operator and $L_x$ the Lie algebra action, see \cite[VI, \S 8]{GHV}.
Restriction to the basepoint of $X$ yields an isomorphism $I_G \cong C^* (\fg, \fk)$.
By definition, the cohomology of $C^*(\fg, \fk)$ is the relative Lie algebra cohomology $H^{*}_{\bR} (\fg,\fk; \bC)$ (the subscript indicates that the ground field that Lie algebra cohomology depends on), and because $I_G$ has trivial differential, we get an isomorphism 
\begin{equation}\label{invariant-forms}
I_G \cong H^* (\fg,\fk). 
\end{equation}

\begin{remark}
In the following we assume that $G$ acts \emph{trivially} on $I_G$. The group $\Sp_{2g}(\bR)$ is connected and so in this case there is no problem. In the case $G=\Ort_{g,g}(\bR)$, the index $2$ subgroup $\SO_{g,g}(\bR)$ has two connected components, the unit component is denoted $\SO_{g,g}(\bR)^+$. Since $\Ort_{g,g}$ is an algebraic group, the action on the (finite-dimensional) vector space $I_G$ is algebraic as well. But the group $\SO_{g,g} (\bR)^+ $ acts trivially and so does its Zariski closure, which is $\SO_{g,g}(\bR)$. The whole group $\Ort_{g,g}(\bR)$ does not act trivially, but we will circumvent this problem below by an \emph{ad hoc} argument.
\end{remark}

By complexification, we obtain isomorphisms
\begin{equation}\label{lie-algebra-cohomology}
H^{*}_{\bR} (\fg,\fk; \bC)  \cong H^{*}_{\bC} (\fg\otimes \bC,\fk\otimes \bC; \bC) \cong H^{*}_{\bC} (\fl\otimes \bC,\fk\otimes \bC) \cong H^{*}_{\bR} (\fl,\fk; \bC),
\end{equation}
where we have indicated the ground field in the notation for Lie algebra cohomology; the middle isomorphism is given by an isomorphism $\fg \otimes \bC\cong \fl \otimes \bC$ extending the identity on $\fk \otimes \bC$. Let $H^{*}_{\sm} (L;\bR)$ be the cohomology of the complex of smooth group cochains on $L$. The van Est spectral sequence in this case is (cf.\ \cite[\S III.7.6]{Gui})
\begin{equation*}
E_{2}^{p,q} = H^{q}_{dR} (L/K; \bC) \otimes H^{p}_{\sm} (L;\bC) \Longrightarrow H^{p+q}_{\bR}(\fl,\fk; \bC).
\end{equation*}
Since $L$ is compact, $H^{p}_{sm} (L)=0$ for $p>0$ and we get an isomorphism
\begin{equation}\label{cohomology-homogeneous-space}
H^{*}_{\bR} (\fl,\fk; \bC) \cong H^{*}_{dR} (L/K; \bC).
\end{equation}
Combining the isomorphisms (\ref{cohomology-homogeneous-space}), (\ref{lie-algebra-cohomology}) and (\ref{invariant-forms}), we can rewrite the Borel--Matsushima homomorphism as 
\begin{equation}\label{borel-matsushima2}
\tilde{j}^*:H^{*}(L/K;\bC) \cong H^{*}_{\bR}( \fl;\fk;\bC) \cong I_{G} \stackrel{j^*}{\lra} H^{*}(B\Gamma;\bC).
\end{equation}
We can further reinterprete this. Let $\lambda:L/K \to BK$ be the classifying map of the $K$-principal bundle $L \to L/K$ and $\psi: B\Gamma \to BG$ induced by the inclusion. Moreover, let $\mu: BG \to BK$ be a homotopy inverse to the induced map $BK \to BG$. Composing $H^* (BK)\to H^{*}(L/K)$ with (\ref{borel-matsushima2}), we obtain a diagram
\begin{equation}\label{borelmatsushima-topological}
\begin{gathered}
\xymatrix{
  H^* (BK;\bC) \ar[d]^{\cong}_{\mu^*} \ar[r]^{\lambda^* } & H^*(L/K;\bC) \ar[d]^{\tilde{j}^*} \\
H^* (BG;\bC)  \ar[r]^{\psi^*} & H^* (B \Gamma;\bC).
}
\end{gathered}
\end{equation}

\begin{proposition}\label{borel-mathsushima3}
The diagram \eqref{borelmatsushima-topological} is commutative.
\end{proposition}

This was shown by Borel, see \cite[Proposition 7.2]{Borel4} (and \cite[\S 3.2]{Gian} for more details). It is this topological version that is most convenient for our purposes. 

\subsection{Borel's main theorems}

Now we discuss the ranges where Borel proved that $j^*$ (and hence $\tilde{j}^*$) is an isomorphism. 

Let $\bG$ be a connected semisimple algebraic group defined over $\bZ$ (e.g.\ $\bG = \SO_{g,g}$ or $\bG=\Sp_{2g}$). Let $G=\bG(\bR)$, $G_{\bC}= \bG(\bC)$, and $K$ and $L$ be as in the previous subsection. Let $\Gamma \subset G$ be an arithmetic subgroup, i.e.\ a subgroup such that $\Gamma \cap \bG(\bZ)$ has finite index in both $\Gamma$ and $\bG(\bZ)$. Let $r: \bG(\bC) \to \GL(E)$ be a finite-dimensional holomorphic representation (which is automatically algebraic). The Borel--Matsushima homomorphism can be extended to a map
$$j^q: H^q (\fg,\fk;E) \lra H^q (\Gamma;E).$$

\begin{theorem}[Borel]\label{borelstheorem}
There is a constant $c(\bG,r)$ such that
\begin{enumerate}[(i)]
\item  $j^q$ is injective for $q \leq c(\bG,r)$ and surjective if in addition $q \leq \rank_{\bR} (\bG)$. 
\item If $E^G=0$, then $H^q (\Gamma;E)=0$ for $q \leq \min (c(\bG;r), \rank_{\bR} (\bG)-1)$.
\end{enumerate}
In our cases of interest, these constants may be estimated as follows.
\begin{enumerate}[(i)]
\item If $\bG=\SO_{g,g}$, and if $r$ is the $k$th tensor power of the defining representation $V$, then $c(\bG,r) \geq g-2-k$.
\item If $\bG=\Sp_{2g}$, and if $r$ is the $k$th tensor power of the defining representation $V$, then $c(\bG,r) \geq g-1-k$.
\end{enumerate}
\end{theorem}

\begin{proof}
What remains to be done is the computation of $c(\bG,r)$, and there is a recipe for that, see \cite{Borel2}. We discuss the symplectic case as the orthogonal case is very similar. The constant $c(\bG,r)$ can be read off from the root system of $G$. Pick a maximal $\bQ$-split torus. Pick a system $\Phi^+$ of positive roots and let $\fn$ be the Borel-subalgebra. Let $\rho$ be half the sum of positive roots. For each character $\mu \in \chi(S)$, let $c(G;\mu)$ be the maximum of all $q$ such that for each weight $\eta$ of $S$ in $\Lambda^q \fn$, we have $\rho-\mu-\nu >0$, i.e. it is a sum of positive roots. For a rational representation $r$, let $c(G,r)$ be the minimum of all $c(G,\mu)$, where $\mu$ is a weight  of $r$ with respect to $S$.

To do this in a concrete example, consider $\Sp_{2n}$. It is of type $C_g$. The above data in this concrete case are described in \cite[p.\ 338]{Proc}.
A basis for the Cartan subalgebra is given by the matrices $e_{i,i}-e_{g+i,g+i}$, $i=1,\ldots,g$, which identifies the Cartan subalgebra with $\bR^n$. Let $\alpha_i$ be the $i$th coordinate function. The roots are $\pm(\alpha_i+\alpha_j)$, $i<j$, $\alpha_i-\alpha_j$, $i \neq j$ and $\pm 2 \alpha_i$, $i=1,\ldots,g$. The system of positive roots is given by $\alpha_i + \alpha_j; \alpha_i-\alpha_j$ ($i<j$) and $2 \alpha_i$, $i=1,\ldots,g$. The simple roots are $\alpha_i-\alpha_{i+1}$, $i=1,\ldots,g-1$ and $2 \alpha_g$. A linear form that distinguishes the positive from the negative roots is given by $L:\alpha_i \mapsto a_i$; $a_1 > \cdots > a_g >0$ and we may take $a_1 = e^{10g}$, $a_j=e^{-10 j g}$ to ease the estimates that follow.
The weight $\rho$ is 
$$\rho = \sum_{i=1}^{g} (g-i+1) \alpha_i.$$

The weights of $\fn$ are precisely the positive roots; thus the weights of $\Lambda^q \fn$ are the sums of $q$ different positive roots. The weights of the defining representation are $\pm \alpha_i$, $i=1,\ldots,g$ and those of its $k$th tensor power are therefore the sums of $k$ of those.  

\begin{claim}
If $r$ is the $k$th tensor power of the defining representation, then $c(G,r) \geq g-1-k$.
\end{claim}

We have to prove that $g-1-k \leq c(G,\mu)$ for all weights of $r$, in other words, if $q \leq g-1-k$, $\mu$ is a weight of $r$ and $\eta$ a sum of $q$ distinct positive roots, then $\rho - \mu -\eta >0$ or $L(\rho - \mu -\eta) >0$. Due to our choice of $L$, it is easy to see that
$$L(\rho - \mu -\eta) \geq L(\rho - k\alpha_1 -\eta) \geq L\left(\rho - r \alpha_1 - \sum_{j=1}^{q} ( \alpha_1 + \alpha_j)\right).$$
According to the value of $\rho$, this is
$$\sum_{i=1}^{g} (g-i+1) a_i - ka_1 - \sum_{j=1}^{q} (a_1 + a_j) = (g-k-q-1)a_1 +  \sum_{i=2}^{g} (g-i+1) a_i - \sum_{i=2}^{q}   a_i.$$
The second and third summand yield
$$\sum_{i=2}^{g} (g-i+1) a_i - \sum_{i=2}^{q}  a_i = \sum_{i=2}^{q} (g-i)a_i + \sum_{i=q+1}^{g} (g-i+1)a_i \geq (g-q)\sum_{i=2}^{q} a_i + \sum_{i=q+1}^{g} a_i >0$$
if $q < g$. Thus, if $g-k-1 \geq q$, the whole sum is positive, as claimed.
\end{proof}

\subsection{Proof of Theorem \ref{cohomology-arithmetic-group-constant-coeffs}: cohomology of $\Ort_{g,g}(\bZ)$, $\Sp_{2g}(\bZ)$ and $\Gamma_g(1,2)$}

We now show how Theorem \ref{cohomology-arithmetic-group-constant-coeffs} follows from the results surveyed in this section. By the universal coefficient theorem, it is enough to prove this for cohomology with coefficients in $\bC$.
Consider first the group $\SO_{g,g} (\bZ)$ instead of $\Ort_{g,g}(\bZ)$. Then we take (with notation from the beginning of this section) $K= \mathrm{S}(\Ort(g)\times \Ort(g))$, $L=\SO(2g)$. According to Proposition \ref{borel-mathsushima3} the diagram
\begin{equation*}
\begin{gathered}
\xymatrix{
H^* (B\Ort;\bC) \ar[r]^-{\eta^*} \ar[rd]_{\lambda^* \eta^*} & H^* (B\mathrm{S}(\Ort(g) \times \Ort(g);\bC) \ar[r]^-{\psi^* \mu^*}  \ar[d]^{\lambda^*} &  H^*(B \SO_{g,g} (\bZ);\bC) \\
 & H^* \left(\frac{\SO(2g)}{\mathrm{S}(\Ort(g) \times \Ort(g))};\bC\right) \ar[ur]_{\tilde{j}^*} & 
}
\end{gathered}
\end{equation*}
is commutative, where $\tilde{j}^*$ is the Borel--Matsushima homomorphism or rather its version \ref{borel-matsushima2} (which is, by Theorem \ref{borelstheorem}, an isomorphism in the range of degrees under consideration) and all other maps are induced by the maps of spaces introduced before. To prove that $(\beta^{4m}_{g})^* = (\eta \circ \mu \circ \psi)^*$ is an isomorphism in the stable range, it is therefore enough to show that $(\eta \circ \lambda)^*$ is an isomorphism in the stable range. But 
$$\frac{\SO(2g)}{\mathrm{S}(\Ort(g) \times \Ort(g))} \cong \frac{\Ort(2g)}{\Ort(g) \times \Ort(g)}$$
and it is a classical fact that the map
$$\frac{\Ort(2g)}{\Ort(g)\times \Ort(g)} \overset{\lambda}{\lra} B\Ort(g) \times B\Ort(g) \overset{\eta}\lra B\Ort$$
is $(g-1)$-connected. This proves that the composition 
$$B \SO_{g,g}(\bZ) \lra B \Ort_{g,g}(\bZ) \lra B\Ort$$
induces an isomorphism in cohomology with complex coefficients in degrees $\leq g-2$, and so on cohomology with rational or real coefficients also. In particular, the map $H^* (B\Ort_{g,g}(\bZ);\bR) \to H^* (B\SO_{g,g}(\bZ);\bR)$ is surjective in this range of degrees. On the other hand, $p:B \SO_{g,g}(\bZ) \to B \Ort_{g,g}(\bZ)$ is up to homotopy a twofold covering and thus, by the classical transfer argument, induces an injection in real cohomology in all degrees. This finishes the proof of Theorem \ref{cohomology-arithmetic-group-constant-coeffs} in this case, and gives the following corollary.

\begin{corollary}\label{cor:TrivGaloisAction}
The action of the Galois group $\bZ/2$ of the cover acts trivially on $H^*(B \SO_{g,g}(\bZ);\bR)$ in degrees $* \leq g-2$.
\end{corollary}

The proof in the case of $\Sp_{2g}(\bZ)$ or $\Gamma_g(1,2)$ is similar; let us write $\Gamma$ for either of these groups. The relevant diagram in this case is
\begin{equation*}
\begin{gathered}
\xymatrix{
H^* (\SO /\U;\bC) \ar[r]^-{\eta^*}  &  H^* (B\U(g);\bC) \ar[r]^-{\psi^* \circ \mu^*}  \ar[d]^-{\lambda^*} &  H^*(B \Gamma;\bC) \\
 & H^* \left(\frac{\USp(g)}{\U(g)};\bC\right) \ar[ur]_-{\tilde{j}^*}. & 
}
\end{gathered}
\end{equation*}
Here, $\lambda: \USp(g)/\U(g) \to B\U(g)$ is the classifying map for the bundle $\USp(g) \to \USp(g)/\U(g)$, and $\tilde{j}^*$ is the Borel--Matsushima homomorphism, which is an isomorphism in degrees $* \leq g-1$ by Theorem \ref{borelstheorem} (note that $\Gamma_g(1,2) \leq \Sp_{2g}(\bZ) = \bG(\bZ)$ has finite index, so arithmetic, hence satisfies the hypotheses of that theorem). The triangle commutes by Proposition \ref{borel-mathsushima3}.

It remains to prove that $(\eta \circ \lambda)^*$ is an isomorphism in degrees $* \leq g-1$. The rational homotopy groups of $\SO/\U$ have rank 1 in degrees congruent to 2 modulo 4 and are trivial otherwise, and the same is true for the rational homotopy groups of $\USp(g) /\U(g)$ in degrees $<2g$. Look at the maps
$$\pi_{4k+2} (\USp(g)/\U(g)) \otimes \bQ\stackrel{\lambda_*}{\to}  \pi_{4k+2} (B\U(g))\otimes \bQ \stackrel{\eta_*}{\to} \pi_{4k+2} (\SO/\U) \otimes \bQ \stackrel{}{\to} \pi_{4k+2} (B\U)\otimes \bQ.$$
The first map is injective since $\pi_{4k+2} (\USp(g))\otimes \bQ =0$. The composition of the other two maps is $\Delta_*$. An easy calculation with Chern classes shows that
$$\pi_{4k+2}(\Delta) :\pi_{4k+2}(B\U(g)) \lra \pi_{4k+2}(B\U)$$
is multiplication by 2, so rationally injective (recall that we are in the stable range, as $4k+2 \leq g-1$) and therefore $\eta_* \circ \lambda_*$ is rationally injective and hence, since source and target are both $1$-dimensional, a rational isomorphism. Hence the induced map $(\eta \circ \lambda)^*$ in rational cohomology is an isomorphism in degrees $* \leq g-1$ (in fact, roughly up to degree $2g$). The same then holds for cohomology with complex coefficients too.

\subsection{Arithmetic representations and cohomology}

\begin{definition}\label{def-arithm-repres}
Let $\bG$ be $\Sp_{2g}$ or $\Ort_{g,g}$ and $\Gamma \subset \bG(\bR)$ be an arithmetic subgroup. A finite dimensional complex representation $\Gamma \to \GL (E)$ is called \emph{arithmetic} if there exists a holomorphic representation $\bG(\bC) \to \GL(E)$ extending $r$. If there is moreover a $\bG(\bC)$-equivariant injection $E \to (V^{\otimes k})^m$ for some $m \in \bN$ (where $V$ denotes the defining representation of $\bG(\bC)$), we say that $E$ is \emph{arithmetic of degree $\leq k$}.
\end{definition}

\begin{lemma}\label{lemma-properties-artihmetic-reps}
Let $\bG$ be as in Definition \ref{def-arithm-repres}, let $\Gamma \subset \bG(\bR)$ be an arithmetic group and assume that it is Zariski dense.
\begin{enumerate}[(i)]
\item Each arithmetic representation has finite degree.
\item Each arithmetic representations of $\Gamma$ is a direct sum of irreducibles.
\item Any subrepresentation $F \subset E$ of an arithmetic representation is arithmetic and the degree of $F$ is bounded by the degree of $E$. 
\item If $F \subset E$ is a subrepresentation of an arithmetic representation, then $E/F$ is arithmetic and the degree of $E/F$ is bounded by the degree of $E$. 
\end{enumerate}
\end{lemma}

\begin{proof}
Let $K \subset \bG(\bC)$ be a maximal compact subgroup. There is a bijection between unitary representations of $K$ and holomorphic representations of $\bG(\bC)$, \cite[Corollary 8.7.1]{Proc}. By the Peter--Weyl theorem for linear groups \cite[Theorem 8.32]{Proc}, any unitary representation of $K$ is contained in $(V^{\otimes k})^m$ for some integers $k,m$, and thus the same is true for a holomorphic representation of $\bG(\bC)$. This shows the first claim.

The assumption on Zariski density implies that $\Gamma$ is Zariski dense in $\bG(\bC)$. Let $E$ be an arithmetic representation and $F \subset E$ be a $\Gamma$-invariant subspace. Let $d$ be the dimension of $F$ and $\Gr_d (E)$ be the Grassmannian of $d$-dimensional subspaces, which is a projective variety. The group $\bG(\bC)$ acts algebraically on $\Gr_d (E)$ and $\Gamma$ fixes the point $F\in \Gr_d (E) $. As $\Gamma$ is assumed to be Zariski dense, the group $\bG(\bC)$ fixes $F$ as well, which means that $F$ is a $\bG(\bC)$-subrepresentation. Clearly, the action of $\bG(\bC)$ on $F$ is still holomorphic. This proves the third claim (the degree bound is obvious).

Because any holomorphic representation of $\bG(\bC)$ decomposes into a sum of irreducibles, the above argument shows that $F \subset E$ has a $\bG(\bC)$-invariant, and so $\Gamma$-invariant, complement. Therefore the second claim holds. The fourth statement follows from the second one because $F$ has a $\Gamma$-invariant complement in $E$ (which is isomorphic to $E/F$).
\end{proof}

\begin{proposition}\label{prop:corollary-to-borel}
Assume that $\bG = \Sp_{2g}$, $\bG=\SO_{g,g}$ or $\bG = \Ort_{g,g}$, let $G:= \bG(\bR)$ and $\Gamma \subset G$ be arithmetic and Zariski dense.
Let $r: \Gamma \to \GL(E)$ be an arithmetic representation of degree $\leq k$. Assume that $q \leq g-2-k$ (orthogonal case) or $q \leq g-1-k$ (symplectic case). Then the inclusion $E^G \to E$ induces isomorphisms
$$H^q (\Gamma; \bC) \otimes E^{\Gamma}  =H^q (\Gamma; \bC) \otimes E^G \cong  H^q (\Gamma; E^G) \cong H^q (\Gamma; E).$$
\end{proposition}
\begin{proof}
The first equality holds by Zariski density, the second is clear.
By Lemma \ref{lemma-properties-artihmetic-reps}, we can decompose the representation $E$ as a sum of irreducibles; thus, without loss of generality, $E$ is irreducible and by Lemma \ref{lemma-properties-artihmetic-reps}, we can assume that the degree of $E$ is $\leq k$ as well.
If $E$ carries the trivial $G$-action, the claim is a tautology, so we can assume that $E^{G} =0$. Theorem \ref{borelstheorem} finishes the proof when $\bG = \Sp_{2g}$ or $\bG=\SO_{g,g}$.

If $\bG = \Ort_{g,g}$, put $\Gamma_0 := \Gamma \cap \SO_{g,g}(\bR)$, which is an index $2$ subgroup of $\Gamma$ and is arithmetic and Zariski dense as a subgroup of $\SO_{g,g}(\bR)$. We have already shown above that $H^* (\Gamma_0; E) \cong H^* (\Gamma_0;\bC) \otimes E^{\SO_{g,g}}$. Taking $\bZ/2$-invariants on both sides, we get isomorphisms
$$(H^* (\Gamma_0;\bC) \otimes E^{\SO_{g,g}} )^{\bZ/2}\cong H^* (\Gamma_0; E)^{\bZ/2}\cong H^* (\Gamma;E),$$
but we have shown in Corollary \ref{cor:TrivGaloisAction} that the group $\bZ/2$ acts trivially on the (real, and hence complex) cohomology of $\Gamma_0$ in degrees $\leq g-2$, and so as $(E^{\SO_{g,g}} )^{\bZ/2} = E^{\Ort_{g,g}}$ the claimed result follows.
\end{proof}

\begin{lemma}\label{gamma-group:zariski-dense}
The groups $\Gamma_{g} (1,2) $ and $\Sp_{2g}(\bZ)$ are Zariski dense in $\Sp_{2g}(\bC)$.
\end{lemma}

\begin{proof}
It is clear that $\Sp_{2g}(\bZ)$ is Zariski dense in $\Sp_{2g}(\bC)$ . Let $\Gamma\subset \Sp_{2g}(\bZ)$ be the intersection of all conjugates of $\Gamma_g (1,2)$ under $\Sp_{2g}(\bZ)$; this is a normal subgroup of finite index. The Zariski closure $\bar{\Gamma} \subset \Sp_{2g}(\bC)$ is a Zariski-closed normal subgroup (and thus analytically closed). Since $\Sp_{2g}(\bC)$ is connected and simple, $\bar{\Gamma}$ is either finite (and contained in the centre) or all of $\Sp_{2g}(\bC)$. Clearly, it is the second of these alternatives which is true.
\end{proof}

\section{Application of rational homotopy theory and surgery theory}\label{section:rational-homotopy}

The purpose of this and the next section is to prove the following result.

\begin{theorem}\label{arithmeticity-theorem}
In the stable range, the $\Gamma(W_g^{2n})$-representation $H^q (B \Tor_{g, 1}^{2n};\bQ)$ is arithmetic of degree $\leq q$, for each $q \in \bN_0$.
\end{theorem}

In this section, we follow the arguments by Berglund--Madsen \cite{BM}. We will be working with semi-simplicial groups and semi-simplicial monoids, so let us first establish some basic constructions for them.

If $M$ is a (possibly topological) monoid, recall that there is a functorial construction of a contractible free $M$-space $EM$ such that $BM := EM/M$ is a delooping of $M$ (provided that $\pi_0 (M)$ is a group, which is always the case for the monoids we consider here). Let $M_\bullet$ be a semi-simplicial monoid, and $EM_\bullet$ be the semi-simplicial space formed by applying the construction $M \mapsto EM$ levelwise. Let $BM_\bullet$ be the semi-simplicial space obtained by forming the construction $M \mapsto BM$ levelwise, i.e.\ $BM_\bullet = EM_\bullet / M_\bullet$, where the quotient is taken levelwise. Write $BM := \vert BM_\bullet \vert$. If $f_\bullet : M_\bullet \to N_\bullet$ is a map of semi-simplicial monoids, let $N \hcoker M$ denote the homotopy fibre of $Bf : BM \to BN$. This has an  action of the Moore loop space $\Lambda BN$, and so an action of $\pi_1(BN)$ in the homotopy category, which gives an action of $\pi_0(\vert N_\bullet \vert)$ on $N \hcoker M$ in the homotopy category.
If $f_{\bullet}$ is an inclusion, we let $(N/M)_{\bullet}:=N_{\bullet}/M_{\bullet}$ be the levelwise quotient, and there is a natural map $|(N/M)_{\bullet}| \to N \hcoker M$ which is a weak equivalence.

We introduce the following simplified notation. Let $M$ be a compact smooth manifold with boundary.

\begin{definition}
\mbox{}
\begin{enumerate}[(i)]
\item $\blockdiff_{\partial}(M)$ is the semi-simplicial group of block diffeomorphisms of $M$. Its $p$-simplices are the diffeomorphisms $\phi$ of $M \times \Delta^p$ which fix $\partial M \times \Delta^p$ and which preserve the face structure of $\Delta^p$, i.e.\ for each $\sigma \subset \Delta^p$, $\phi$ restricts to a diffeomorphism of $M \times \sigma$; the $i$th face map is given by restriction of a diffeomorphism to the $i$th face of $\Delta^p$. There is an inclusion $\mathrm{Sing}_\bullet^{sm}\Diff_{\partial}(M)$ of the semi-simplicial set of smooth singular simplices of the topological group $\Diff_{\partial}(M)$, and $|\mathrm{Sing}_\bullet^{sm}\Diff_{\partial}(M)| \simeq \Diff_{\partial}(M)$.
\item $\blockaut_{\partial}(M)_\bullet$ is the semi-simplicial monoid of block homotopy equivalences: its $p$-simplices are the self homotopy equivalences of $M \times \Delta^p$ which fix $\partial M \times \Delta^p$ pointwise and which preserve the face structure of $\Delta^p$ as above. There is an inclusion $\blockdiff_{\partial}(M)\subset \blockaut_{\partial} (M)$ of semi-simplicial monoids.
\item $D := \mathrm{Sing}_\bullet^{sm}\Diff_{\partial}(W_{g,1}^{2n})$.
\item $\tilde{D} := \blockdiff(W_{g,1}^{2n})_\bullet$.
\item $\tilde{G} := \blockaut_{\partial}(W_{g,1}^{2n})_\bullet$.
\item $\tilde{G}' \subset \tilde{G}$ is the union of those path components which contain vertices of $D$.
\item $\Gamma := \Gamma(W_{g,1}^{2n})$, which we consider as a semi-simplicial group which only has simplices in degree zero.
\end{enumerate}
\end{definition}
We have natural homomorphisms
\begin{equation*}
\xymatrix{
D \ar[r]& \tilde{D} \ar[r]& \tilde{G}' \ar[r]\ar[d]& \tilde{G}\\
& & \Gamma
}
\end{equation*}
between these semi-simplicial groups and monoids, and $B \Tor_{g, 1}^{2n} \simeq \Gamma \hcoker D$ in this notation. In this section we will study the block analogue of this space, $\Gamma \hcoker \tilde{D}$, and prove the following proposition analogous to Theorem \ref{arithmeticity-theorem} for it. In the following section we will show how to deduce Theorem \ref{arithmeticity-theorem} from this proposition.

\begin{proposition}\label{block-arithmeticity-theorem}
For $g \geq 2$ and $q \leq (n-1)$, the $\Gamma$-representation $H^q (\Gamma \hcoker \tilde{D};\bQ)$ is arithmetic of degree $\leq q$.
\end{proposition}

There is a fibration sequence
\begin{equation}\label{eq:FibSeq}
\tilde{G}' \hcoker \tilde{D} \overset{i}\lra \Gamma \hcoker \tilde{D} \overset{p}\lra \Gamma \hcoker \tilde{G}'
\end{equation}
and we claim

\begin{lemma}\label{lem:inj-Gamma-rep}
For $g \geq 2$ and $q \leq (n-1)$, the map $i^*: H^q(\Gamma \hcoker \tilde{D} ;\bQ) \to H^q(\tilde{G}' \hcoker \tilde{D};\bQ)$ is an injective map of $\pi_0(\vert D_\bullet\vert)$-modules.
\end{lemma}

\begin{proof}
The map $i$ is $\pi_0(\vert\tilde{G}'_\bullet\vert)$-equivariant in the homotopy category, so in particular $\pi_0(\vert D_\bullet\vert)$-equivariant in the homotopy category, and we deduce that $i^*$ is a map of $\pi_0(\vert D_\bullet\vert)$-modules. Next, we claim that $ \Gamma \hcoker \tilde{G}'$ is path-connected and has finite homotopy groups in degrees $* \leq (n-1)$, which follows from computations by Berglund and Madsen. More precisely, the fibre sequence $ \Gamma \hcoker \tilde{G}' \to B \tilde{G}' \to B \Gamma$ induces exact sequences 
$$\pi_2 (B \Gamma )=0 \lra \pi_1 (\Gamma \hcoker \tilde{G}') \lra \pi_0 (\tilde{G}') \lra \Gamma \lra \pi_0(\Gamma \hcoker \tilde{G}') \lra \pi_0( B \tilde{G}' )=*.$$
By Proposition \ref{torelli-space--connected} and the definition of $\tilde{G}'$ the map $\pi_0 (\tilde{G}') \to \Gamma$ is surjective, and by \cite[Theorem 2.12]{BM} it has finite kernel, which proves the claim in degrees $* \leq 1$. That the homotopy groups $\pi_* (\Gamma \hcoker \tilde{G}')$ are finite for $ 2 \leq * \leq (n-1)$ follows from \cite[Theorem 2.10]{BM} (as long as $g \geq 2$). By the Leray--Serre spectral sequence for \eqref{eq:FibSeq} we find that the map
$$H^*(\Gamma \hcoker \tilde{D} ;\bQ) \lra H^*(\tilde{G}' \hcoker \tilde{D};\bQ)^{\pi_1(\Gamma \hcoker \tilde{G}')}$$
is an isomorphism for $* \leq(n-1)$, and in particular the map $i^*$ is injective in this range of degrees.
\end{proof}

Note that the natural map $\tilde{G}' \hcoker \tilde{D} \to \tilde{G} \hcoker \tilde{D}$ is simply the inclusion of the basepoint component, so we also denote it by $(\tilde{G}' \hcoker \tilde{D})_0$. In order to study this space we will use surgery theory in the form of Quinn's surgery fibration sequence, which we now briefly review. 

\vspace{2ex}

\emph{Assume that $M$ is simply-connected and of dimension $d \geq 5$.} The surgery fibration (see \cite{Quinn}, \cite{Nicas}) has the form
\begin{equation}\label{quinnsfibration}
\widetilde{\bS}^{} (M, \partial M) \stackrel{\eta}{\lra} \bN(M, \partial M) \stackrel{\sigma}{\lra} \bL (M).
\end{equation}
The \emph{$L$-theory space} $\bL(M)$ has homotopy groups given by Wall's surgery obstruction groups, which as we have supposed $M$ is simply-connected are 
\begin{equation*}
\pi_k (\bL(M)) = L_{k+d} (\bZ \pi_1 (M)) =
\begin{cases}
\bZ &  k + d \equiv 0 \pmod 4\\
\bZ/2 & k+d \equiv 2 \pmod 4\\
 0 & k+d  \equiv 1 \pmod 2.
\end{cases}
\end{equation*}

The \emph{block structure space} $\widetilde{\bS}^{}(M, \partial M)$ is a classifying space for smooth block bundles equipped with a fibre homotopy equivalence to the trivial $M$-bundle. Precisely, let us write $\bR^n_+ = [0,\infty) \times \bR^{n-1}$ and define $\widetilde{\cS}^{}(M, \partial M ; n)_\bullet$ to be the semi-simplicial set with $p$-simplices given by pairs
\begin{enumerate}[(i)]
\item\label{it:S:1} $E \subset  \Delta^p \times \bR^n_+$ a $(d+p)$-dimensional smooth manifold (with corners), such that the projection $\pi : E \to \Delta^p$ is transverse to all faces. The boundary of $E$ is decomposed into $\pi^{-1}(\partial \Delta^p)$ and the closure of its complement, $\partial^v E $ which we require to be $E \cap (\Delta^p \times \{0\} \times \bR^{n-1})$. If $\sigma \subset \Delta^p$ is a face we write $\partial_\sigma E := \pi^{-1}(\sigma)$.

\item\label{it:S:3} A map $\phi : E \to \Delta^p \times M$ such that $\phi\vert_{\partial^v E} : \partial^v E \to \Delta^p \times \partial M$ is a diffeomorphism, and which for every face $\sigma \subset \Delta^p$ restricts to a homotopy equivalence from $\partial_\sigma E$ into $\sigma \times M$.
\end{enumerate}

We remark explicitly that the projection $\pi : E \to \Delta^p$ and the first component of $\phi$ are \emph{not} required to coincide. The face maps are given by restricting the data to codimension 1 faces of $\Delta^p$. There are semi-simplicial inclusions $\widetilde{\cS}^{}(M, \partial M ; n)_\bullet \to \widetilde{\cS}^{}(M, \partial M ; n+1)_\bullet$ given by considering a submanifold $E$ of $\Delta^p \times \bR^{n}$ as lying in $\Delta^p \times \bR^{n+1}$, and we write $\widetilde{\cS}^{}(M, \partial M)_\bullet$ for the colimit. Then $\widetilde{\bS}^{}(M, \partial M)$ is defined to be the geometric realisation $\vert \widetilde{\cS}^{}(M, \partial M)_\bullet \vert$. The block structure space $\widetilde{\bS}^{}(M, \partial M)$ has a left action by the (discrete) group $\Diff(M)^\delta$ of diffeomorphisms of $M$, where a diffeomorphism $\rho$ acts on $p$-simplices by $(E, \phi) \mapsto (E, (\mathrm{Id}_{\Delta^p} \times \rho) \circ \phi)$.

\begin{lemma}\label{structure-space-vs-homogeneous-space}
There is a map
$$c : \widetilde{\bS}^{}(M, \partial M)_0 \lra (\blockaut(M)/\blockdiff(M))_0$$
between basepoint components which is a homotopy equivalence, and is $\Diff(M)^\delta$-equivariant.
\end{lemma}
\begin{proof}
We define the map $c$ simplicially as follows: for $(E, \phi) \in \widetilde{\cS}^{}(M, \partial M)_p$ in the component of the basepoint, it follows from the $h$-cobordism theorem that there exists a diffeomorphism $\psi: E \cong \Delta^p \times M$ (which for each face $\sigma \subset \Delta^p$ restricts to a diffeomorphism of $\partial_\sigma E$ to $\sigma \times M$), and then $\phi\circ\psi^{-1} : \Delta^p \times M \to \Delta^p \times M$ is an element of $\blockaut(M)_p$. However, it depends on our choice of $\psi$: making another choice changes the element we get by right multiplication with a diffeomorphism $\Delta^p \times M \cong \Delta^p \times M$ (which for each face $\sigma$ restricts to a diffeomorphism of $\sigma \times M$), so an element of $\blockdiff(M)_p$. Thus we get a well-defined $p$-simplex of $\blockaut(M)/\blockdiff(M)$, and it is easy to see that this construction defines a simplicial map. The map described lands in the basepoint component, and it is a further easy consequence of the $h$-cobordism theorem that it is a homotopy 
equivalence to this path component. It is furthermore clear from the definition that $c$ is $\Diff(M)^\delta$-equivariant.
\end{proof}

The space of \emph{normal invariants} $\bN(M, \partial M)$ is a classifying space for degree 1 normal maps from a smooth manifold to $M$. Precisely, let ${\cN}(M, \partial M ; n)_\bullet$ be the semi-simplicial set with $p$-simplices given by tuples consisting of a manifold $E$ as in (\ref{it:S:1}) above, as well as
\begin{enumerate}[(i)]
\setcounter{enumi}{1}
\item\label{it:N:3} A map $\phi : E \to \Delta^p \times M$ such that $\phi\vert_{\partial^v E} : \partial^v E \to \Delta^p \times \partial M$ is a diffeomorphism, and which for every face $\sigma \subset \Delta^p$ restricts to a map from $\partial_\sigma E$ into $\sigma \times M$ which has degree 1 (in homology relative to the boundary).

\item\label{it:N:4} An $(n-d)$-dimensional vector bundle $\xi \to \Delta^p \times M$ and a vector bundle map $\hat{\phi} : \nu_E \to \xi$ covering $\phi$, where $\nu_E \to E$ is the normal bundle of $E \subset \Delta^p \times \bR^{n}$.
\end{enumerate}
The face maps are given by restricting the data to codimension 1 faces of $\Delta^p$. There are semi-simplicial inclusions ${\cN}^{}(M, \partial M ; n)_\bullet \to {\cN}^{}(M, \partial M ; n+1)_\bullet$ given by $(E, \phi, \xi, \hat{\phi}) \mapsto (E, \phi, \xi \oplus \epsilon^1, \hat{\phi} \oplus \epsilon^1)$. We write ${\cN}^{}(M, \partial M)_\bullet$ for the colimit, and ${\bN}^{}(M, \partial M)$ is defined to be the geometric realisation $\vert {\cN}^{}(M, \partial M)_\bullet \vert$.

The space of normal invariants has a left action by the (discrete) group $\Diff(M)^\delta$ of diffeomorphisms of $M$, where a diffeomorphism $\rho$ acts on $p$-simplices by 
$$(E, \phi, \xi, \hat{\phi}) \longmapsto (E, (\mathrm{Id}_{\Delta^p} \times \rho)\circ \phi, (\rho^{-1})^*\xi, \hat{\rho}\circ \hat{\phi}),$$
where $\hat{\rho} : \xi \to (\rho^{-1})^*\xi$ is the canonical bundle map covering $\rho$. In the models just described it does not seem to be possible to give a simplicial model of the map $\eta$ which is $\Diff(M)^\delta$-equivariant (the construction of $\eta$ in \cite{Quinn} and \cite{Nicas} make choices which are not natural in this sense). However, it is possible to construct an auxiliary space $\widetilde{\bS}^{}(M, \partial M)'$ and maps
$$\widetilde{\bS}^{}(M, \partial M) \overset{\sim}\longleftarrow \widetilde{\bS}^{}(M, \partial M)' \overset{\eta'}\lra \bN(M, \partial M)$$
which are $\Diff(M)^\delta$-equivariant, and this is an adequate substitute. Briefly, a simplex in $\widetilde{\bS}^{}(M, \partial M)'$ should in addition have a choice of datum (\ref{it:N:4}), that is a vector bundle $\xi \to \Delta^p \times M$ and bundle map $\hat{\phi} : \nu_E \to \xi$ covering $\phi$. Up to homotopy this is no further data, because $\xi$ must be isomorphic to  $g^*(\nu_E)$ where $g$ is a homotopy inverse to $\phi$, and $\hat{\phi}$ must be equivalent to the bundle map induced from the map $\nu_E \to \phi^*g^*\nu_E$ induced by a homotopy $g \circ \phi \leadsto id$. The map $\eta'$ is then simply the inclusion of a subspace.

The space $\bN(M, \partial M)$ is homotopy equivalent to $\mathrm{map}_*(M/\partial M, \mathrm{G}/\Ort)$, see e.g. \cite{Quinn}. Here $\mathrm{G} = \colim_{k \to \infty} \mathrm{G}(k)$ and $\mathrm{G}(k)$ denotes the grouplike topological monoid of self homotopy equivalences of $S^{k-1}$; there is a homomorphism $\Ort(k) \to \mathrm{G}(k)$ by the action of the orthogonal group on the unit sphere, and in the colimit this gives a homomorphism $\Ort \to \mathrm{G}$. The space $\mathrm{G}/\Ort$ is by definition the homotopy fibre of the induced map $B\Ort \to B\mathrm{G}$ on classifying spaces.

\begin{remark}
We believe that there is a zig-zag of $\Diff(M)^\delta$-equivariant homotopy equivalences between $\bN(M, \partial M)$ and the mapping space $\map_*(M/\partial M, \mathrm{G}/\Ort)$, on which $\rho \in \Diff(M)^\delta$ acts by $- \circ \rho^{-1}$, but giving a detailed proof of this would lead us too far afield. Instead, in the following we use a trick which is implicit in \cite{BM}, though we try to explain it in more detail. 
\end{remark}

We have a homotopy equivalence $\kappa: \bN(M, \partial M) \to \mathrm{map}_*(M/\partial M, \mathrm{G}/\Ort)$, and at the level of homotopy groups there is a good description of what this map does: it takes a degree 1 normal map to $D^k \times M$ relative boundary, and associates to it a vector bundle with a fibre homotopy equivalence from its sphere bundle to the Spivak normal fibration of $D^k \times M$ (suitably trivialised over the boundary of $D^k \times M$). For a $\rho \in \Diff(M)^\delta$ we obtain the following diagram
\begin{equation*}
\xymatrix{
\pi_k(\bN(M, \partial M), \mathrm{Id}_M) \ar[d]^-\rho \ar[r]^-\sim &  \pi_k(\mathrm{map}_*(M/\partial M, \mathrm{G}/\Ort) , *) \ar[dd]^-{- \circ \rho^{-1}}\\
\pi_k(\bN(M, \partial M), \rho) \ar[d]^-{\sim}& \\
\pi_k(\bN(M, \partial M), \mathrm{Id}_M) \ar[r]^-\sim& \pi_k(\mathrm{map}_*(M/\partial M, \mathrm{G}/\Ort) , *)
}
\end{equation*}
In this diagram the lower left-hand map is a change of basepoint isomorphism: the space $\bN(M, \partial M)$ is simple and so this is well-defined. The simplicity of $\bN(M, \partial M)$ is because it is homotopy equivalent to $\mathrm{map}_*(M/\partial M, \mathrm{G}/\Ort)$ and because $\mathrm{G}/\Ort$ is an infinite loop space. Furthermore, this diagram can be shown to commute (the fundamental calculation is a formula for the normal invariant of the composition of two homotopy equivalences, cf.\ \cite[Lemma 3.3]{BM}).

Hence $\Diff(M)^\delta$ acts on $\pi_k(\bN(M, \partial M), \mathrm{Id}_M)$, via the geometric action on the space level followed by translation of loops back to the basepoint $\mathrm{Id}_M$, and this $\Diff(M)^\delta$-module is identified with $\pi_k(\mathrm{map}_*(M/\partial M, \mathrm{G}/\Ort) , *)$ with the left action whereby $\rho$ acts by precomposition with $\rho^{-1}$.

\vspace{2ex}

Let us now return to the manifolds $W_{g,1}^{2n}$, and denote them by $W$ for simplicity.

\begin{lemma}\label{cohomology-space-normal-invariants}
The action of the group $D^\delta := \Diff(W)^\delta$ on $H^q (\bN(W, \partial W))_0;\bQ)$ is through the homomorphism $D^\delta \to \Gamma$, and as a $\Gamma$-module it is arithmetic of degree $\leq q$.
\end{lemma}

\begin{proof}
Compare \cite[\S 3.3]{BM}. $\mathrm{G}/\Ort$ is an infinite loop space, so $\map_* (W/\partial W;\mathrm{G}/\Ort)_0$ is too. The rational cohomology of $\map_* (W/\partial W;\mathrm{G}/\Ort)_0$ is thus the symmetric algebra on the dual of the rational homotopy, so it follows that the same is true of $\bN(W, \partial W))_0$. Thus to know the representations $H^q(\bN(W, \partial W))_0;\bQ)$ we merely need to know the representations $\pi_q(\bN(W, \partial W))_0, \mathrm{Id}_W)\otimes \bQ$, which by the above discussion is isomorphic as a representation to $\pi_q(\map_* (W/\partial W;\mathrm{G}/\Ort)_0, *)\otimes \bQ$, so it suffices to calculate these. 

$\mathrm{G}/\Ort$ is connected and has rational homotopy $\bQ$ in positive degrees divisible by $4$ and $0$ otherwise. Therefore $\map_* (W/\partial W;\mathrm{G}/\Ort)_0$ has rational homotopy in positive degrees given by, with the abbreviations $V:= H^n (W;\bQ)$ and $S_k:=\pi_k (\mathrm{G}/\Ort)\otimes \bQ$,
$$\pi_k (\map_* (W/\partial W;\mathrm{G}/\Ort)_0) \cong (V \otimes S_{k+n}) \oplus S_{k+2n}.$$
This isomorphism is $D^\delta$-equivariant, which shows that the $D^\delta$-action factors through $\Gamma$, and the arithmeticity statement and the estimation of the degree follows from this.
\end{proof}

\begin{lemma}\label{lemma:rationally-simple}
Let $X$ be a connected space, $Y$ a connected infinite loop space of finite type and $f:X \to Y$ be a map such that $f$ is injective on rational homotopy in degrees above 1 and has finite kernel on $\pi_1$. Then $f^* :H^* (Y,\bQ) \to H^* (X; \bQ)$ is surjective.
\end{lemma}

We will now give the proof of Proposition \ref{block-arithmeticity-theorem}, and defer the proof of this (standard) lemma for a moment.

\begin{proof}[Proof of Proposition \ref{block-arithmeticity-theorem}]
We restrict the surgery fibration \eqref{quinnsfibration} to the unit components and obtain
$$ \widetilde{\bS}^{} (W,\partial W)_0 \overset{\eta}{\lra} \bN(W,\partial W)_0\overset{\sigma}{\lra} \bL(M)_0$$
because $\bL(M)_0$ is simply connected. By \cite[Lemma 3.4]{BM}, the map $\sigma$ is surjective on rational homotopy groups. Therefore $\eta$ is injective on rational homotopy groups in degrees $\geq 2$, and has finite kernel on $\pi_1$. The induced map on cohomology
$$\eta^* : H^q(\bN(W,\partial W)_0;\bQ) \lra H^q(\widetilde{\bS}^{} (W,\partial W)_0;\bQ)$$
is $D^\delta$-equivariant by our discussion, and we deduce from Lemma \ref{lemma:rationally-simple} that it is surjective. 
Thus the $D^\delta$-action on
$$H^q(\widetilde{\bS}^{} (W,\partial W)_0;\bQ) \stackrel{\text{Lemma \ref{structure-space-vs-homogeneous-space}}}{\cong} H^q((\tilde{G} \hcoker \tilde{D})_0;\bQ) \overset{\sim}\lra H^q(\tilde{G}' \hcoker \tilde{D};\bQ)$$
factors through $\Gamma$ and as such is an arithmetic representation of degree $\leq q$, by Lemma \ref{cohomology-space-normal-invariants}. By Lemma \ref{lem:inj-Gamma-rep} it follows that for $q \leq  (n-1)$ the $\Gamma$-representation $H^q(\Gamma \hcoker \tilde{D};\bQ)$ is arithmetic of degree $\leq q$, which proves the proposition.
\end{proof}

\begin{proof}[Proof of Lemma \ref{lemma:rationally-simple}]
As $Y$ is an infinite loop space, it has a rationalisation $Y_\bQ$ which is unambiguous. By composing $f$ with the rationalisation map $Y \to Y_{\bQ}$, the conclusion of the Lemma is unchanged. Since $Y$ has finite type, the map $Y \to Y_{\bQ}$ has finite kernel on all homotopy groups, and so the hypotheses are also unchanged (except that $Y$ is now of finite rational type). We may therefore assume that $Y$ is a rational infinite loop space, in particular, a generalised Eilenberg--MacLane space $K (W)$ where $W:=\pi_* (Y)\otimes \bQ$. Let $V \subset W$ be the subspace generated by the image of $\pi_* (f)$ and pick a splitting $W \to V$, which induces a map $q:K(W) \to K(V)$. The composition 
$$X \stackrel{f}{\lra} K(W) \stackrel{q}{\lra} K(V)$$
satisfies the assumptions of the lemma, has the additional property that the image $\pi_* f$ spans the rational homotopy of the target, and if $q \circ f$ is surjective in cohomology, then so is $f$. Thus we may in addition assume that the image of $\pi_*(f)$ spans the rational homotopy of $Y$ (and $\pi_1(f)$ still has finite kernel).

Consider the map of fibrations
\begin{equation*}
\xymatrix{
\widetilde{X} \ar[r] \ar[d] & \ar[d] \widetilde{Y} \\
X \ar[r]^{f} \ar[d] & Y \ar[d]\\
B \pi_1 (X) \ar[r] & B \pi_1 (Y).
}
\end{equation*}
The map of universal covers is a rational homotopy equivalence, and so $\pi_1(X)$ acts on the higher homotopy groups $\pi_i(X)\otimes \bQ \cong \pi_i(Y)\otimes \bQ$ through the homomorphism $\pi_1(X) \to \pi_1(Y)$. As $Y$ is an infinite loop space this action is trivial, so the space $X$ is ``rationally simple" in the sense that $\pi_1 (X)$ acts trivially on the higher rational homotopy groups, as well as on $H^*(\widetilde{X};\bQ)$. The map of Leray--Serre spectral sequences for these two fibrations on $E_2$-terms is thus
$$H^p(\pi_1(X);\bQ) \otimes H^q(\widetilde{X};\bQ) \lra H^p(\pi_1(Y);\bQ) \otimes H^q(\widetilde{Y};\bQ),$$
and the map $\pi_1 (X) \to \pi_1 (Y)$ is a homomorphism with finite kernel whose target is a rational vector space and whose image generates the target: thus it induces an isomorphism on rational cohomology. Hence we have an isomorphism of $E_2$-terms, and so get an isomorphism of $E_{\infty}$-terms, proving the lemma.
\end{proof}

Finally, we will complete the proof of Theorem \ref{arithmeticity-theorem}. Consider the fibration sequence
\begin{equation*}
\tilde{D} \hcoker D \lra \Gamma \hcoker D \overset{p}\lra \Gamma \hcoker \tilde{D},
\end{equation*}
where the map $p$ is $\Gamma$-equivariant. We will show that the fibre $\tilde{D} \hcoker D$ has trivial rational homology in the stable range (i.e.\ degrees $* \leq C_g^{2n}$). Since $\tilde{D}\hcoker D$ is obviously connected, the Leray--Serre spectral sequence shows that $H^* (\Gamma \hcoker \tilde{D};\bQ) \to H^* (\Gamma \hcoker D;\bQ)$ is an isomorphism in this stable range and so Theorem \ref{arithmeticity-theorem} follows from Proposition \ref{block-arithmeticity-theorem}.

\begin{theorem}\mbox{}
\begin{enumerate}[(i)]
\item The natural map $\blockdiff(D^{2n})/\Diff(D^{2n}) \to \blockdiff(W^{2n}_{g,1})/\Diff(W^{2n}_{g,1}) =\tilde{D} \hcoker D$ is $(2n-4)$-connected.
\item If $2n \geq \max\{2k+7,3k+4\}$ then $\pi_k(\blockdiff(D^{2n})/\Diff(D^{2n}))$ is finite.
\end{enumerate}
\end{theorem}
\begin{proof}
The first assertion is a consequence of Morlet's lemma of disjunction, and is given in \cite[Corollary 3.2]{BLR}. The second assertion follows from \cite{FHs} or \cite[\S 6.1]{WW2}.
\end{proof}

\section{The final spectral sequence argument}\label{section:final-argument}

Recall that in Theorem \ref{mt-spectra-index} we have constructed a homotopy-commutative square
\begin{equation}\label{eq:MainDiagram}
\begin{gathered}
\xymatrix{
 B \Diff (W_{g}^{2n},D) \ar[d]^{\alpha_{g}^{2n}} \ar[r]^{\zeta} & B \Gamma(W_{g}^{2n}) \ar[d]^{\beta_{g}^{2n}}\\
 \loopinf{}_{0} \MTthetan \ar[r]^{\symb } & \loopinf{+k_n}_{0} \KO,
}
\end{gathered}
\end{equation}
where the vertical arrows are rational cohomology equivalences in the stable range. Next, we take the rationalisation of the lower row and obtain
\begin{equation}\label{eq:MainDiagram1}
\begin{gathered}
\xymatrix{
 B \Diff (W_{g}^{2n},D) \ar[d]^{\alpha_{g}^{2n}} \ar[r]^{\zeta} & B \Gamma(W_{g}^{2n}) \ar[d]^{\beta_{g}^{2n}}\\
 \loopinf{}_{0} \MTthetan_{\bQ} \ar[r]^{\symb_{\bQ} } & \loopinf{+k_n}_{0} \KO_{\bQ}
}
\end{gathered}
\end{equation}
and both spaces in the lower row are connected and simply connected. We wish to compare the Leray--Serre spectral sequences of the rows, so we replace the diagram with an equivalent map which is a map of fibration sequences, by Lemma \ref{coherence} below.
\begin{lemma}\label{coherence}
Let 
\begin{equation*}
\begin{gathered}
\xymatrix{
E \ar[r]^{p} \ar[d]^{f} & B \ar[d]^{g}\\
F \ar[r]^{q} & A
}
\end{gathered}
\end{equation*}
be a homotopy commutative diagram. Then we can replace it (in the homotopy category of diagrams) by a strictly commutative diagram with $p$ and $q$ fibrations.
\end{lemma}

\begin{proof}
Let us write $\hat{q}: F^f := F \times_A A^I \to A$ for the standard path-space replacement of $F \to A$ by a Hurewicz fibration. Choose a homotopy $H : q\circ f \leadsto g\circ p$ which exhibits the original square as commuting up to homotopy. Lifting the homotopy $H$ starting at $E \overset{f}\to F \hookrightarrow F^f$ gives a map $\hat{f}\vert_E : E \to F^f$ such that $\hat{q} \circ \hat{f}\vert_E = g \circ p$. We have produced an equivalent square which strictly commutes, but $p$ is not necessarily a fibration. However, if we write $\hat{p}: E^f := E \times_B B^I \to B$ for the standard path-space replacement of $E \to B$, then $\hat{f}\vert_E : E \to F^f$ extends to a map $E^f \to F^f$ over $g$ by sending a pair $(e \in E, \gamma : p(e) \leadsto x)$ to the pair $(\hat{f}\vert_E(e), g\circ \gamma : g\circ p(e) \leadsto g(x))$.
\end{proof}

This lemma yields a commutative diagram 
\begin{equation}\label{eq:MainDiagram2}
\begin{gathered}
\xymatrix{
B \Tor_{g,1}^{2n} \ar[r] \ar[d]  &  B \Diff (W_{g}^{2n},D) \ar[d]^{\alpha_{g}^{2n}} \ar[r]^{\zeta} & B \Gamma(W_{g}^{2n}) \ar[d]^{\beta_{g}^{2n}}\\
\Omega^{\infty} F \ar[r] & \loopinf{}_{0} \MTthetan_{\bQ} \ar[r]^{\symb_{\bQ} } & \loopinf{+k_n}_{0} \KO_{\bQ},
}
\end{gathered}
\end{equation}
where the rows are fibration sequences, $\Omega^{\infty} F$ is by definition the homotopy fibre of $\symb_{\bQ}$ (which is an infinite loop map), and the left vertical map depends on the choice of the homotopy making the square \eqref{eq:MainDiagram1} commute.

Consider first the top row of \eqref{eq:MainDiagram2}. We have seen that $H^q (B \Tor_{g,1}^{n};\bC)$ is an arithmetic $\Gamma (W_{g,1}^{2n})$-representation of degree $\leq q$, in the stable range. By Proposition \ref{prop:corollary-to-borel} and Lemma \ref{gamma-group:zariski-dense}, we have an isomorphism
$$H^p (\Gamma (W_{g,1}^{2n});H^q (B \Tor_{g,1}^{2n};\bC)) \cong H^p (\Gamma (W_{g,1}^{2n});\bC) \otimes H^q (B \Tor_{g,1}^{2n};\bC))^{\Gamma (W_{g,1}^{2n})},$$
as long as $q$ is in the concordance stable range and $p \leq g-2-q$. 

As the lower row of \eqref{eq:MainDiagram} is a fibration of infinite loop spaces with simply connected base, its spectral sequence collapses and the coefficient systems are constant. Thus the comparison map is a map of spectral sequences of algebras that have a product structure (the first one only in the concordance stable range). The $E^{2}_{*,0}$-map is an isomorphism by Theorem \ref{borelstheorem} and the target map is an isomorphism by Galatius and Randal-Williams' theorem. By the Zeeman comparison theorem \cite[Theorem 2]{Zeeman} the map on the $E^{2}_{0,*}$ line is also an isomorphism, and so the natural map
$$H^q (\Omega^{\infty}F;\bC) \lra H^q (B \Tor_{g,1}^{2n};\bC)^{\Gamma (W_{g,1}^{2n})}$$
is an isomorphism too, in the various stable ranges.

The computation of the map $\symb$ in rational cohomology (Theorem \ref{mt-spectra-index}) shows that $H^* (\Omega^\infty F; \bQ)$ is equal to the source of the natural map in Theorem \ref{main-result}, and this finishes the proof of Theorem \ref{main-result}.
The quotient of the algebra (\ref{eq:cohom-calc}) by the ideal generated by the elements $\kappa_{\cl_i}$ is then $\bQ[\lambda_{\B{i}},\mu_{\B{j}} \,\,\vert\,\, \B{i},\B{j} \in \bN_{0}^{I}, w (\B{i})>2n, w(\B{j}) >0, |\B{i}| \geq 2]$, and Theorem \ref{main-result} may be stated in the following equivalent form.

\begin{theorem}\label{thm:StableCalc}
The natural map 
\begin{equation*}
\bQ[\lambda_{\B{i}},\mu_{\B{j}} \,\,\vert\,\, \B{i},\B{j} \in \bN_{0}^{I}, w (\B{i})>2n, w(\B{j}) >0, |\B{i}| \geq 2] \lra H^* (B \Tor_{g, 1}^{2n};\bQ)^{\Gamma (W_{g, 1}^{2n})}
\end{equation*}
is an isomorphism in degrees $* \leq C_{g}^{2n}$.
\end{theorem}

We will now show how to deduce Theorem \ref{main-result:calc} from this form of Theorem \ref{main-result}.

\begin{proof}[Proof of Theorem \ref{main-result:calc}]
We are only interested in the stable range $C_g^{2n} \leq n-3$. The classes $\mu_{\B{j}}$ have degree $w(\B{j}) \geq n+1$ so do not contribute to the stable range. The classes $\lambda_{\B{i}}$ have degree $w(\B{i})-2n$, so in Theorem \ref{thm:StableCalc}  only those with $\vert \B{i} \vert = 2$ can possibly occur in the stable range. These are exactly the classes $\kappa_{\cl_a \cl_b}$ with $\tfrac{n+1}{4} \leq a \leq b$, as claimed in Theorem \ref{main-result:calc}.
\end{proof}

\section{Relation to classical invariant theory}\label{sec:InvThy}

The work of Berglund and Madsen \cite{BM} may be used to compute the rational cohomology of $B\Tor_{g,1}^{2n}$ in the stable range as a $\Gamma(W_{g,1}^{2n})$-module, and so the right hand side of the map of Theorem \ref{main-result} may also be approached by the invariant theory of the groups $\Sp_{2g}$ and $\Ort_{g,g}$. In this short section we briefly compare these approaches.

\begin{proposition}
Let $P_*$ be the graded vector space obtained from $\pi_*(G/O) \otimes \bQ$ by shifting degrees down by $n$ and then discarding all terms which do not have strictly positive grading. There is an isomorphism of graded $\Gamma(W_{g,1}^{2n})$-modules
$$H^*(B\Tor_{g,1}^{2n};\bQ) \cong S[V \otimes P_*]$$
in degrees $* < n-1$, where $V$ is the fundamental representation of $\Sp_{2g}(\bZ)$ or $\Ort_{g,g}(\bZ)$, and $S$ denotes the free graded-commutative algebra.
\end{proposition}
\begin{proof}
By \cite[Theorem 3.5]{BM} the rational cohomology of $\tilde{G}' \hcoker \tilde{D}$ as a $\pi_0(\tilde{G}')$-module is identified with the free algebra in the statement of the theorem, so in particular the $\pi_0(\tilde{G}')$-action on $H^*(\tilde{G}' \hcoker \tilde{D};\bQ)$ factors through $\Gamma$. We now consider the fibration sequence \eqref{eq:FibSeq}, where we have shown that we obtain an isomorphism
$$H^*(\Gamma \hcoker \tilde{D} ;\bQ) \lra H^*(\tilde{G}' \hcoker \tilde{D};\bQ)^{\pi_1(\Gamma \hcoker \tilde{G}')}$$
in degrees $* < n-1$. But the action of $\pi_1(\Gamma \hcoker \tilde{G}')$ on $H^*(\tilde{G}' \hcoker \tilde{D};\bQ)$ is through $\pi_1(B\tilde{G}')=\pi_0(\tilde{G}')$ which we have seen acts through $\Gamma$, and is hence the $\pi_1(\Gamma \hcoker \tilde{G}')$-action is trivial. Thus $H^*(\Gamma \hcoker \tilde{D} ;\bQ) \lra H^*(\tilde{G}' \hcoker \tilde{D};\bQ)$ is an isomorphism in degrees $* < n-1$.
\end{proof}

Let us consider Theorem \ref{main-result:calc} in the case $n=4k$, with $k \gg 0$ and in degrees $* < n$. Let us rename the class $\kappa_{\cl_{k+a}\cl_{k+b}}$, with $a \leq b$ and both at least 1, by $\rho_{a,b}$, which has degree $4(a+b)$. Then Theorem \ref{thm:StableCalc} implies that the map
$$\bQ[\rho_{a,b} \,\vert\, a, b \in \{ 1,\ldots , 3k \}] \lra \mathrm{Sym}^\bullet[V \otimes(\bQ[4] \oplus \bQ[8] \oplus \bQ[12] \oplus \cdots)]^{\Ort_{g,g}}$$
is an isomorphism in cohomological degrees $* \leq \frac{g-5}{2}$ (as we have supposed that $k$, and so $n$, is very large). We obtain as a corollary the following weak form of the fundamental theorem of invariant theory.

\begin{corollary}
In degrees $\bullet \leq \frac{g-3}{8}$, the ring of invariants $\mathrm{Sym}^\bullet[V]^{\Ort_{g,g}}$ is polynomial generated by $\omega \in \mathrm{Sym}^2[V]$, the element representing the (dual of the) pairing.
\end{corollary}
\begin{proof}
By the isomorphism obtained above, we see that $\mathrm{Sym}^\bullet[V \otimes(\bQ[4] \oplus \bQ[8] \oplus \bQ[12] \oplus \cdots)]^{\Ort_{g,g}}$ is a polynomial ring concentrated in degrees divisible by 4, with generators in cohomological degree $4i$ given by the $\rho_{a,b}$ with $a+b=i$. Thus the generators in degree $4i$ are given by partitions of $i$ into two parts of sizes $\{ 1,\ldots , 3k \}$, so as we have supposed $k \gg 0$, by partitions of $i$ into two proper parts.

On the other hand, the canonical invariant $\omega \in \mathrm{Sym}^2[V]$ gives many invariants in $\mathrm{Sym}^2[V \otimes(\bQ[4] \oplus \bQ[8] \oplus \bQ[12] \oplus \cdots)]^{\Ort_{g,g}}$; for each pair of integers $x \leq y$ in $\{4, 8, 12, \ldots\}$ we have an invariant $\omega_{x,y} \in (V \otimes \bQ[x]) \otimes (V \otimes \bQ[y])$ of cohomological degree $x+y$. Thus in cohomological degree $4i$ we have as many $\omega_{x,y}$ as the number of partitions of $i$ into two proper parts. It is clear that $\mathrm{Sym}^1[V \otimes(\bQ[4] \oplus \bQ[8] \oplus \bQ[12] \oplus \cdots)]$ has no invariants, so the $\omega_{x,y}$ are all indecomposable, and they are also clearly linearly independent. Hence by counting dimensions
$$\bQ[\omega_{x, y}\,\vert\, x \leq  y \in \{4, 8, 12, \ldots\}] \lra \mathrm{Sym}^\bullet[V \otimes(\bQ[4] \oplus \bQ[8] \oplus \bQ[12] \oplus \cdots)]^{\Ort_{g,g}}$$
is an isomorphism in cohomological degrees $* \leq \frac{g-3}{2}$. By observation the intersection with the subring $\mathrm{Sym}^\bullet[V\otimes \bQ[4]]^{\Ort_{g,g}}$ is $\bQ[\omega_{4,4}]$; changing from cohomological degrees $*$ to symmetric-power degrees $\bullet$ gives a factor of 4, so we obtain the statement in the corollary.
\end{proof}

One may improve the range of degrees to $\bullet \leq \frac{g-3}{4}$ by considering $n=4k+2$ instead, but the counting arguments are a little more complicated. One may also consider $n=4k+1$ to show that $\Lambda^\bullet[V]^{\Sp_{2g}}$ is polynomial on the canonical class in $\Lambda^2[V]$ dual to the pairing, for $\bullet \leq \frac{g-3}{2}$.

\bibliographystyle{alpha}

\end{document}